\documentclass[a4paper,11pt]{amsart}

\usepackage{amssymb}
\usepackage{mathrsfs}
\usepackage{amsmath,amssymb,amsthm,latexsym,amscd,mathrsfs}
\usepackage{indentfirst}
\usepackage{stmaryrd}
\usepackage{graphicx}
\usepackage{extarrows}
\usepackage{multirow}
\usepackage{latexsym}
\usepackage{amsfonts}
\usepackage{color}
\usepackage{pictexwd,dcpic}
\usepackage{graphicx}
\usepackage{psfrag}
\usepackage{hyperref}
\usepackage{comment}
\usepackage{enumerate}

\numberwithin{equation}{section} \theoremstyle{plain}
\newtheorem{thm}{Theorem}[section]

\newtheorem{prop}[thm]{Proposition}
\newtheorem{lem}[thm]{Lemma}
\newtheorem{cor}[thm]{Corollary}
\newtheorem{defn}{Definition}
\newtheorem{conj}{Conjecture}

\newtheorem{rem}[thm]{Remark}

\newtheorem{ques}{Question}

\newtheorem{ack}{Acknowledgements}   
 \makeatletter

\def\<{\langle}
\def\>{\rangle}
\def\({\left(}
\def\){\right)}
\def\[{\left[}
\def\]{\right]}

\DeclareMathOperator{\diag}{diag}
\DeclareMathOperator{\rank}{rank}
\DeclareMathOperator{\Tr}{Tr}
\def\Re{\mathop{\text{Re}}}
\DeclareMathOperator{\oi}{\mathbf{i}}
\DeclareMathOperator{\Span}{Span}
\def\Vec{\mathop{\text{Vec}}}
\makeatother
\title{On some conjectures by Lu and Wenzel}

\author[J.Q. Ge]{Jianquan Ge}
\address{School of Mathematical Sciences, Laboratory of Mathematics and Complex Systems, Beijing Normal University, Beijing 100875, P.R. CHINA.}
\email{jqge@bnu.edu.cn}

\author[F.G. Li]{Fagui Li}
\address{School of Mathematical Sciences, Laboratory of Mathematics and Complex Systems, Beijing Normal University, Beijing 100875, P.R. CHINA.}
\email{faguili@mail.bnu.edu.cn}

\author[Z.Q. Lu]{Zhiqin Lu}
\address{Department of Mathematics, University of California,
Irvine, Irvine, CA 92697, USA.}
\email{zlu@math.uci.edu}

\author[Y. Zhou]{Yi Zhou}
\address{School of Mathematical Sciences, Laboratory of Mathematics and Complex Systems, Beijing Normal University, Beijing 100875, P.R. CHINA.}
\email{zhou\_yi@mail.bnu.edu.cn}

\date{}
\thanks{The first author is partially supported by the NSFC (No. 11522103, 11331002) and by the Fundamental Research Funds for the Central Universities of China. The third author is partially supported by the National Sciences Foundation of USA (DMS-19-08513).}

\begin{document}
\maketitle

% % % % % % % % % %
\begin{abstract}
In order to give a unified generalization of the BW inequality and the DDVV inequality, Lu and Wenzel proposed three
Conjectures \ref{conj1}, \ref{conj2}, \ref{conj3} and an open Question \ref{Lu-Wenzel ques} in 2016.
In this paper we discuss further these conjectures and put forward several new conjectures which will be shown equivalent to Conjecture \ref{conj2}. In particular, we prove Conjecture \ref{conj2} and hence all conjectures in some special cases. For Conjecture \ref{conj3}, we obtain a bigger upper bound $2+\sqrt{10}/2$, and we also give a weaker answer for the more general Question \ref{Lu-Wenzel ques}. In addition, we obtain some new simple proofs of the complex BW inequality and the condition for equality.
\end{abstract}

%%%%%%%%%%%%%%%%%%%%%%%
\section{Introduction}\label{sect-intr}
In 2005, B\"ottcher and Wenzel \cite{BW05} raised the so-called BW conjecture that if $X$, $Y$ are real square matrices, then $$\|XY-YX\|^2\leq2\|X\|^2\|Y\|^2,$$ where $\|X\|=\sqrt{\Tr XX^*}$ is the Frobenius norm (here $X^*$ is the conjugate transpose of $X$). For real $2\times2$ matrices, the proof was obtained   by B\"ottcher and Wenzel in \cite{BW05}, and L\`aszl\`o \cite{L07} proved the $3\times3$ case.
The first proof for the real $n\times n$ case was found by Vong and Jin \cite{VJ08} and independently by Lu \cite{Lu11}. After that B\"ottcher and Wenzel found another proof (cf. \cite{BW08, W10}) that also extends to the case of complex matrices. Then immediately
Audenaert \cite{AKMR09} gave a simplified proof by probability method and Lu \cite{Lu12} also got a different simple proof by eigenvalue method. The complete characterization of the equality was given in \cite{CVW10} and another unitarily invariant norm attaining the minimum norm bound for commutators was given in \cite{FKC10}.
Some generalizations of the BW-type inequalities were obtained by Wenzel and Audenaert \cite{WKMRA10}, also by Fong, Lok, Cheng \cite{CFL13} and Cheng, Liang \cite{CL17}.

In comparison with the BW inequality that estimates the Frobenius norm of the commutator between two arbitrary matrices, the DDVV inequality estimates the Frobenius norm of the commutators among arbitrary many real symmetric matrices.
Recall that the DDVV inequality comes from the normal scalar curvature conjecture (DDVV conjecture) in submanifold geometry posed by De Smet, Dillen, Verstraelen and Vrancken \cite{DDVV99} in 1999:
Let $M^n\rightarrow N^{n+m}(\kappa)$ be an isometric immersed $n$-dimensional submanifold in the real space form with constant sectional curvature $\kappa$.
Then there is a pointwise inequality $$\rho+\rho^\bot\leq\|H\|^2+\kappa, $$
where $\rho$ is the scalar curvature (intrinsic invariant), $H$ is the mean curvature vector field and $\rho^\bot$ is the normal scalar curvature (extrinsic invariants).
Dillen, Fastenakels and Veken \cite{DFV07} then transformed this conjecture into an equivalent algebraic version (DDVV inequality):
\begin{equation*}
\sum^m_{\alpha,\beta=1}\|\[B_\alpha,B_\beta\]\|^2\leq c \(\sum^m_{\alpha=1}\|B_\alpha\|^2\)^2,
\end{equation*}
here  $c=1$ when $B_1,\cdots,B_m$ are real $n\times n$ symmetric matrices.
There were many researches on the DDVV conjecture (cf. \cite{DHTV04, CL08, GT11, Lu07} etc.).
Finally Lu \cite{Lu11} and Ge-Tang \cite{GT08} proved the DDVV inequality (and hence the DDVV conjecture) independently and differently. After then various of DDVV-type inequalities were obtained such as: $c=\frac{1}{3}$ ($n=3$) and $c=\frac{2}{3}$ ($n\geq4$) for real skew-symmetric matrices (cf. \cite{Ge14}); $c=\frac{4}{3}$ for Hermitian matrices (cf. \cite{GXYZ17}) and also for arbitrary real or complex matrices (cf. \cite{GLZ18}).

With the BW inequality and the DDVV inequality on both hands, Lu and Wenzel (\cite{LW16, LW17}) summarized the commutator estimates and considered a unified generalization of them. They proposed the following three conjectures and an open question. Let
$M(n,\mathbb{K})$ be the space of $n\times n$ matrices in the field $\mathbb{K}$.

\begin{conj}\label{conj1}
Let $B_1,\cdots,B_m\in M(n,\mathbb{R})$ be real $n\times n$ matrices subject to  $$\Tr \Big(B_\alpha [B_\gamma, B_\beta]\Big)=0$$ for any $1\leq \alpha,\beta,\gamma\leq m$, then
\begin{equation}\label{csymDDVV}
\sum^m_{\alpha,\beta=1}\|\[B_\alpha,B_\beta\]\|^2\leq  \(\sum^m_{\alpha=1}\|B_\alpha\|^2\)^2.
\end{equation}
\end{conj}
%In particular, it is exactly the famous DDVV conjecture if $B_1,\cdots,B_m$ are real symmetric square matrices.

\begin{conj}\label{conj2}
{$($LW Conjecture$)$.}
Let $B,B_2,\cdots,B_m\in M(n,\mathbb{R})$ be matrices with
\begin{enumerate} [ \rm (i) ]
\item $\Tr (B_\alpha B^*_\beta)=0$  $(i.e., B_\alpha \bot B_\beta)$ \quad for any $\alpha\neq \beta$;
\item $\Tr \Big(B_\alpha [B, B_\beta]\Big)=0$   ~~~~~~~~~~~~~~ \quad\quad for any $2\leq \alpha,\beta\leq m$.
\end{enumerate}
 Then
\begin{equation}\label{ineq of conj2}
\sum^m_{\alpha=2}\|\[B,B_\alpha\]\|^2\leq \(\max \limits_{2\leq \alpha\leq m}\|B_\alpha\|^2+\sum^m_{\alpha=2}\|B_\alpha\|^2\)\|B\|^2.
\end{equation}
\end{conj}

\begin{conj}\label{conj3}
For $X\in M(n,\mathbb{R})$ with $\|X\|=1$,  let $T_X$ be the linear map on $M(n,\mathbb{R})$ defined by $T_X(Y)=\[X^*, \[X, Y\]\]$ and
$\lambda(T_X):=\{\lambda_1(T_X)\geq\lambda_2(T_X)\geq\lambda_3(T_X)\cdots \}$ be the set of eigenvalues of $T_X$. Then
$$\lambda_1(T_X)+\lambda_3(T_X)\leq3.$$
\end{conj}

%[Lu-Wenzel 2015 AMS-MSU Meeting]
\begin{ques}\label{Lu-Wenzel ques}
What is the upper bound of
$\sum\limits_{i=1}^{k}\lambda_{2i-1}(T_X)?$
\end{ques}
If $k = 1$, the bound is $2$ by the BW inequality, i.e., $\lambda_1(T_X)\leq2$, since we have
$$\lambda_1(T_X)=\max \limits_{\|Y\|=1}\langle T_X Y, Y\rangle=\max \limits_{\|Y\|=1}\|[X,Y]\|^2\leq 2.$$
If $k = 2$, the bound is supposed to be $3$ by Conjecture \ref{conj3}.
On the other hand, when restricted to real symmetric matrices, Conjecture \ref{conj1} reduces to the DDVV inequality. It turns out that not only the BW inequality and the DDVV inequality but also both Conjectures \ref{conj1} and \ref{conj3} are implied by Conjecture \ref{conj2} (cf. \cite{LW16}). Moreover, we will show that Conjecture \ref{conj2} is equivalent to assigning $k+1$ as the upper bound of $\sum_{i=1}^{k}\lambda_{2i-1}(T_X)$ for $k\geq1$, which is nothing but the following Conjecture \ref{conj2A} because we can prove $\lambda_{2i-1}(T_X)=\lambda_{2i}(T_X)$ for any $i$ (See Proposition \ref{prop2}). Hence, Conjecture \ref{conj2}, as well as its equivalent Conjectures \ref{conj2A}-\ref{conj2C} in the following, takes exactly the role of a unified generalization of the BW inequality and the DDVV inequality for real matrices. We call Conjecture \ref{conj2} the Fundamental Conjecture of Lu and Wenzel, or simply the (real) LW Conjecture.

\begin{conj}\label{conj2A}
For $X\in M(n,\mathbb{R})$ with $\|X\|=1$, we have
\begin{equation}\label{ineq of conj2A}
\sum_{i=1}^{2k}\lambda_{i}(T_X)\leq2k+2, \quad k=1,\cdots,[\frac{n^2}{2}].
\end{equation}
\end{conj}

In fact, the summation $\sum_{i=1}^{2k}\lambda_{i}(T_X)$ in Conjecture \ref{conj2A} cannot exceed $2n$. We explain this by introducing the following Conjecture \ref{conj2B} which looks stronger but in fact is equivalent to Conjecture \ref{conj2A}. Before that, we introduce some notations.

Let $x=(x_1,x_2,\cdots,x_n)\in\mathbb{R}^n$. We rearrange the components of $x$ in decreasing order and obtain a vector $x^\downarrow=(x_1^\downarrow,x_2^\downarrow,\cdots,x_n^\downarrow)$ where
$$x_1^\downarrow\geq x_2^\downarrow\geq\cdots,\geq x_n^\downarrow.$$

\begin{defn} \cite{Z11}
For $x=(x_1,x_2,\cdots,x_n)$ and $y=(y_1,y_2,\cdots,y_n)$ in
$\mathbb{R}^n$, we say that $x$ is weakly majorized by $y$, written as $x\prec y$, if
$$\sum_{i=1}^kx_i^\downarrow\leq \sum_{i=1}^ky_i^\downarrow,\quad k=1,2,\cdots,n .$$
\end{defn}

\begin{defn} \cite{RAB09}
A multiset may be formally defined as a $2-tuple (A, m)$ where $A$ is the underlying set of the multiset, formed from its distinct elements, and
$\displaystyle m\colon A\to \mathbb {N} _{\geq 1}$ is a function from $A$ to the set of the positive integers, giving the multiplicity, that is, the number of occurrences, of the element $a$ in the multiset as the number $m(a)$.
\end{defn}

If $A=\{a_1,a_2,\dots,a_n\} $ is a finite set, the multiset $(A, m)$ is often represented as $\{a_1^{m(a_1)},a_2^{m(a_2)},\dots,a_n^{m(a_n)}\} $. For example, the multiset $\{a,a,b\}$ is written as $\{a^2,b\}$.

\begin{conj}\label{conj2B}
For $X\in M(n,\mathbb{R})$ with $\|X\|=1$, the set $\lambda(T_X)$ of eigenvalues of $T_X$
 is weakly majorized by the multiset $\{2^{2},1^{2n-4},0^{\(n-1\)^2+1}\}$.
\end{conj}
It is just $$\sum_{i=1}^{2k}\lambda_{i}(T_X)\leq 2n,\quad \textit{for } k\geq n,$$ that looks stronger in the assertion here than in that of Conjecture \ref{conj2A}.
Another equivalent conjecture that also looks stronger is the following Conjecture \ref{conj2C} by omitting the second assumption of Conjecture \ref{conj2}.

\begin{conj}\label{conj2C}
Let $B,B_2,\cdots,B_m\in M(n,\mathbb{R})$ be matrices with $\Tr (B_\alpha B^*_\beta)=0$  for any $2\leq\alpha\neq \beta\leq m$. Then
$$\sum^m_{\alpha=2}\|\[B,B_\alpha\]\|^2\leq \(2\max_{2\leq \alpha\leq m}\|B_\alpha\|^2+\sum^m_{\alpha=2}\|B_\alpha\|^2\)\|B\|^2.$$
\end{conj}

We summarize the relations of these conjectures in the following theorem.

\begin{thm}\label{equiv conj}
%{$($The Relations of Conjectures$)$.}
\begin{enumerate}
\item Conjectures \ref{conj2}, \ref{conj2A}, \ref{conj2B} and \ref{conj2C} are equivalent to each other.
\item If one of the above conjectures is true, then Conjectures \ref{conj1} and \ref{conj3} hold.
\end{enumerate}
%The conjecture \ref{conj2} is equivalent to conjecture \ref{conj2A}, \ref{conj2B} and \ref{conj2C} and conjecture \ref{conj2} is true then also conjecture \ref{conj1} holds.
\end{thm}

%There are some deep connections between these conjectures and inequalities. In other words, these conjectures show that geometric inequalities can be transformed into matrix eigenvalue problems by algebraic inequalities:
%$$\fbox{Geometric\ inequalities}\Longleftrightarrow
% \fbox{Algebraic\ inequalities} \Longleftrightarrow
% \fbox{Matrix\ eigenvalues}$$
%
%On the one hand,
%suppose $X$ is a complex square matrix with $\|X\|$=1, define $T_X$ just as the linear operator in conjecture \ref{conj3}, then
%$$The\ complex\ BW\  conjecture\ holds \Longleftrightarrow \lambda_1(T_X)\leq2.$$
%On the other hand, it is explicit to check that the DDVV conjecture \cite{DDVV99} is the special case of conjecture \ref{conj1}. In fact, let $X$ be an arbitrary real square symmetric matrix with $\|X\|=1$, we will show that
%$$The\ DDVV\  conjecture\ holds \Longleftarrow \sum_{t=1}^{2k}\lambda_{i}(T_X)\leq2k
%+2 \ for\ all\ k\geq1.$$
%
%
%They also raised the following question:

Since the BW inequality (resp.  the DDVV inequality) holds also for complex (resp. complex symmetric) matrices (cf. \cite{BW08}, \cite{GLZ18}), we can also expect for the same conjectures as above with all matrices being complex matrices\footnote{Notice that for the complex version, the vanishing conditions in the conjectures should be in the form of taking trace other than Hermitian inner product, since trace is complex linear while Hermitian inner product is not.}. In fact we will prove the relations of Theorem \ref{equiv conj} between these conjectures in complex version. Hence we call Conjecture \ref{conj2} for complex matrices the complex LW Conjecture. Obviously, the complex LW conjecture implies the real LW conjecture. For example, we restate the complex LW Conjecture in the forms of Conjectures \ref{conj2A} and \ref{conj2B} in the following. Notice that now the map $T_X$ is a self-dual (Hermitian) positive semi-definite operator on the space $M(n,\mathbb{C})$ of complex matrices.
%As a similar question, we propose the following conjecture \ref{conj2D}. Specially, we will prove the  conjecture \ref{conj2D} is correct for $rank(X)=1$ and the equivalent property of conjecture \ref{conj2} and conjecture \ref{conj2D} for all real square matrix $X$. And we also think that conjecture \ref{conj2D} is a common generalization of these inequalities and conjectures.
\begin{conj}\label{conj2D}
{$($Complex LW Conjecture \ref{conj2A}$)$.} For $X\in M(n,\mathbb{C})$ with $\|X\|=1$, we have
$$\sum_{i=1}^{2k}\lambda_{i}(T_X)\leq2k+2, \quad k=1,\cdots,[\frac{n^2}{2}].$$
\end{conj}

\begin{conj}\label{conj4}
{$($Complex LW Conjecture \ref{conj2B}$)$.} For $X\in M(n,\mathbb{C})$ with $\|X\|=1$, the set $\lambda(T_X)$ of eigenvalues of $T_X$
is weakly majorized by the multiset $\{2^{2},1^{2n-4},0^{\(n-1\)^2+1}\}$. %^{\downarrow}$
\end{conj}

In this paper, we prove the complex LW Conjecture (and hence all conjectures posed above) in some special cases which we conclude in the following.
 \begin{thm}\label{thm-special-cases}
 The complex LW Conjectures \ref{conj2D}, \ref{conj4} and hence all conjectures of this paper are true in one of the following cases:
\begin{enumerate} [ \rm (i) ]
 \item $X\in M(n,\mathbb{C})$ is a normal matrix;
 \item $\rank X=1$;
 \item $n=2,3$.
 \end{enumerate}
 \end{thm}

For the conjectures in general we can only get some weaker results as follows.
\begin{thm}\label{thm-conj3-weak}
For $X\in M(n,\mathbb{C})$ with $\|X\|=1$, we have
 $$\lambda_1(T_X)+\lambda_3(T_X)\leq\frac{4+\sqrt{10}}{2}.$$
\end{thm}

\begin{thm}\label{thm-conj7-weak}
For $X\in M(n,\mathbb{C})$ with $\|X\|=1$, we have
$$\sum_{t=1}^{2k}\lambda_{i}(T_X)\leq2k+1+2\sqrt{k},\quad k=1,\cdots,[\frac{n^2}{2}].$$
\end{thm}

It turns out that the methods we developed in the study of the conjectures above lead us to some new simple proofs of the complex BW inequality and the condition for equality, which we will discuss first in Section \ref{sect-BW} as it is just the first eigenvalue estimate $\lambda_1(T_X)\leq2$, the basic case $k=1$ of the complex LW Conjecture \ref{conj2D}. In Section \ref{sect-pre} we prepare several useful lemmas and properties of $T_X$. In Section \ref{sect-eq} we prove the equivalence between Conjectures \ref{conj2A}-\ref{conj2C} and Conjecture \ref{conj2}, i.e., Theorem \ref{equiv conj} in the complex version. In Section \ref{sect-LW} we prove the conjectures for the special cases of Theorem \ref{thm-special-cases} and for general cases, we show the partial results Theorems \ref{thm-conj3-weak} and \ref{thm-conj7-weak}.

Although the inequalities we study in this paper are matrix inequalities, it is not hard to generalize
them as inequalities of bounded operators on separable Hilbert spaces. In quantum physics, these inequalities are related to the \emph{Uncertainty Principle}, or more precisely, the Robertson-Schr\"odinger relations. The classical Uncertainty Principle, in our notations, can be formulated by
\begin{equation*}
\|[A,B]|^2_{OP}\leq 2\,\|A\|^2_{OP}\cdot\|B\|^2_{OP},
\end{equation*}
where $\|\cdot\|_{OP}$ is the operator norm. In this context, the BW-type inequality can be viewed as another version of the Uncertainty Principle. There are literature in physics provides various of generalization of the Uncertainty Principle; see ~\cite{Tri} for example. In our paper, we study the optimal version of all these inequalities.

\section{Preliminaries}\label{sect-pre}
In this section, we will introduce some necessary notations and lemmas which are interesting in themselves. To avoid needless duplication, we discuss the complex version directly so as to include the real version.

Let $T$ be a linear mapping on a complex $N$-dimensional vector space $V$ with Hermitian inner product $\langle\cdot,\cdot\rangle$. In this paper, we always denote by $$\lambda(T):=\{\lambda_1(T)\geq\cdots \geq\lambda_N(T)\}, \quad \sigma(T):=\{\sigma_1(T)\geq\cdots\geq \sigma_N(T)\geq0\}$$ the ordered sets of real eigenvalues (if available) and singular values of $T$ respectively, where singular values are square roots of eigenvalues of $T^*T$.

Now suppose $T\geq0$ be self-dual and positive semi-definite. Then by elementary linear algebra, we have
\begin{lem}\label{lem0}
The multiplicity of each positive eigenvalue of $T$ is even if and only if there exists a unitary skew-symmetric mapping $S$ (i.e., $U^*SU$ is real skew-symmetric for some unitary matrix $U$) such that $T=S^*S=-S^2$. In addition, $Tx=0$ if and only if $Sx=0$.
\end{lem}
\begin{proof}
The sufficiency is clear. Now suppose that there are $g$ distinct positive eigenvalues $\lambda(T)=\{t_1=s_1^2> \cdots> t_g=s_g^2>0\}$  with multiplicities $2n_1, \cdots , 2n_g$, and denote by $r=2\sum_{j=1}^gn_j$ the rank of $T$. Then we can diagonalize $T$ by a unitary matrix $U$ as
$$T=U\diag\Big(t_1I_{2n_1},\cdots,t_gI_{2n_g},O_{N-r}\Big)U^*,$$
where $O_{N-r}$ denotes the zero matrix of order $N-r$.
Then the required unitary skew-symmetric matrix can be defined as  $$S:=U
\begin{pmatrix}\begin{array}{cc}
O & -s_1I_{n_1}\\
s_1I_{n_1}& O
\end{array}& & & \\
&\ddots&&\\&& \begin{array}{cc}
O & -s_gI_{n_g}\\
s_gI_{n_g}& O
\end{array}&\\ &&&O_{N-r}
\end{pmatrix}
U^*.$$

The proof is complete.
\end{proof}

Now let $T$ be self-dual and positive semi-definite with even multiplicities of positive eigenvalues (i.e., $\lambda_{2i-1}(T)=\lambda_{2i}(T)$ for any $i$ with $\lambda_{2i-1}(T)>0$), and $S$ be the unitary skew-symmetric mapping as in Lemma \ref{lem0}. Then we have the following lemmas.

\begin{lem}\label{lemma 1}
Let $y\in V$ with $|y|=1$. Then
\begin{equation*}\label{ineq of lemma 1}
\Big\langle T\frac{Sy}{|Sy|}, \frac{Sy}{|Sy|}\Big\rangle\geq \langle Ty, y \rangle.
\end{equation*}
\end{lem}
\begin{proof}
Since $T=S^*S=-S^2$, the inequality above is equivalent to
$$\<T^2y, y\>\geq\<Ty, y\>^2.$$
Let $\{e_i\}_{i=1}^{N}$ be an orthonormal basis of $V$
such that $e_i$ is a unit eigenvector corresponding to $\lambda_i(T)$.
Setting $y=\sum_{i=1}^{N}y_ie_i$, then $\sum_{i=1}^{N}y_i^2=1$ and we have
$$\begin{aligned}
\<T^2y, y\> & =\sum_{i=1}^{N}y_i^2\lambda_i^2(T)
=\(\sum_{i=1}^{N}y_i^2\lambda_i^2(T)\)\(\sum_{i=1}^{N}y_i^2\)\\
&\geq \(\sum_{i=1}^{N}y_i^2\lambda_i(T)\)^2=\<Ty, y\>^2.
\end{aligned}$$
The proof is complete.
\end{proof}

\begin{lem}\label{lemma 2}
Let $W\subseteq V$ be a complex $m$-dimensional isotropic subspace of $S$, i.e.,
$S(W)\subset W^\bot$ $($ $\langle S w_1, w_2\rangle=0$ for any $w_1,w_2\in W$ $)$.
Then we have $$\Tr T|_W\leq \Tr T|_{S(W)},\quad \Tr T|_W\leq\sum_{i=1}^m\lambda_{2i-1}(T).$$
\end{lem}
\begin{proof}
We will find a suitable basis to compare the traces by using Lemma \ref{lemma 1}.
Let $\{E_i\}_{i=1}^{N}$ be an orthonormal basis of $V$ such that $\{E_i\}_{i=1}^m$ is a basis of $W$, and under this basis we identify $V\cong \mathbb{C}^N$. Denote
$$\rank(SE_1, \cdots, SE_m)=\dim S(W)=: k\leq m.$$
Assume $k\geq 1$, otherwise we have $S|_W=0$ and thus $\Tr T|_W=0$ by Lemma \ref{lem0}.
By singular value decomposition, there exist $P\in U(N)$ and $Q\in U(m)$ such that
$$P^*(SE_1, \cdots, SE_m)Q=\Lambda=:
\begin{pmatrix}
    \widetilde{\Lambda}_{k\times k} & O\\
     O & O
\end{pmatrix}_{N\times m},$$
where $\widetilde{\Lambda}=:\diag(\Lambda_1, \cdots, \Lambda_k)$, $\Lambda_i>0$ for $1\leq i\leq k$.
Setting $$P\Lambda=:(F_1, \cdots, F_m),$$ we have
$\<F_i, F_j\>=\Lambda_i\Lambda_j\delta_{ij}$ for $1\leq i, j\leq k$ and
$F_i=0$ for $i>k$.
Thus $\{\widetilde{F}_i\}_{i=1}^{k}$ is an orthonormal basis of $S(W)$,
where $\widetilde{F}_i:=\Lambda_i^{-1}F_i$.
Let $$(\widetilde{E}_1, \cdots, \widetilde{E}_m):=(E_1, \cdots, E_m)Q,$$ then
$\{\widetilde{E}_i\}_{i=1}^{m}$ is an orthonormal basis of $W$ and satisfies
$$(F_1, \cdots, F_m)=P\Lambda=(SE_1, \cdots, SE_m)Q=(S\widetilde{E}_1, \cdots, S\widetilde{E}_m).$$
Therefore, Lemma \ref{lemma 1} implies
$$\Tr T|_W=\sum_{i=1}^{m}\<T\widetilde{E}_i, \widetilde{E}_i\>
\leq\sum_{i=1}^{k}\<T\widetilde{F}_i, \widetilde{F}_i\>=\Tr T|_{S(W)}.$$

Since $S(W)\subset W^\bot$, $\{\widetilde{E}_i\}_{i=1}^{m}\bigcup\{\widetilde{F}_i\}_{i=1}^{k}$
is an orthonormal basis of $W\oplus S(W)$.
Hence,
$$\Tr T|_W+\Tr T|_{S(W)}=\Tr T|_{W\oplus S(W)}
\leq\sum_{i=1}^{m+k}\lambda_i(T)\leq\sum_{i=1}^{2m}\lambda_i(T),$$
$$\Tr T|_W\leq\frac12\sum_{i=1}^{2m}\lambda_i(T)=\sum_{i=1}^m\lambda_{2i-1}(T).$$
The proof is complete.
\end{proof}

Now we consider the linear operator $T_X$ as in Conjecture \ref{conj3}. More specifically, for any $n\times n$ complex matrix $X$ with $\|X\|=1$, we define
\begin{equation}\label{def-TX}
\begin{aligned}
T_X:  M(n,\mathbb{C}) &\longrightarrow  M(n,\mathbb{C}),\\
 Y &\longmapsto  [X^*,[X,Y]].
\end{aligned}
\end{equation}
 It turns out that $T_X$ is exactly an operator of the same type as $T$ in the preceding lemmas with $V=M(n,\mathbb{C})$, $\dim V=n^2=:N$  (cf. \cite{Lu11}). For the sake of completeness, we repeat the properties as follows.

\begin{prop}\label{prop1}\cite{Lu11}
$T_X$ is an self-dual and positive semi-definite linear map.
\end{prop}
\begin{proof}
This is because of the following straightforward computations:
$$
\langle Y_1,[X^*,[X,Y_2]]\rangle=\langle[X,Y_1],[X,Y_2]\rangle
=\langle [X^*,[X,Y_1]],Y_2\rangle
$$
and $$\langle T_X Y, Y\rangle=\|[X,Y]\|^2.$$
\end{proof}

It follows immediately from the definition (\ref{def-TX}) that
\begin{equation}\label{unitary}
T_{U^*XU}(U^*YU)=U^*(T_XY)U, \quad \textit{for } U\in U(n),
\end{equation}
thus we have
\begin{lem}\label{lem-inv}
The set of eigenvalues $\lambda(T_X):=\{\lambda_1(T_X)\geq\cdots \geq\lambda_N(T_X)\}$ is invariant under unitary congruences of $X$.
\end{lem}

\begin{prop}\label{prop2}\cite{Lu11}
The multiplicity of each positive eigenvalue of $T_X$ is even, i.e., $\lambda_{2i-1}(T_X)=\lambda_{2i}(T_X)$ for any $i$ with $\lambda_{2i-1}(T_X)>0$.
\end{prop}
\begin{proof}
Let $\lambda>0$ be a positive eigenvalue of $T_X$ and $E_\lambda$ be its eigenspace. We will show that the complex dimension of $E_\lambda$ is even.

Define a quasi-linear map by
\begin{equation*}%\label{def-SX}
\begin{aligned}
\widetilde{S}_X:  M(n,\mathbb{C}) &\longrightarrow  M(n,\mathbb{C}),\\
 Y &\longmapsto  [X,Y]^*.
\end{aligned}
\end{equation*}
Then it follows easily that $\widetilde{S}_X (zY)=\bar{z}\widetilde{S}_X (Y)$ for $z\in\mathbb{C}$, $\widetilde{S}_X$ is anti-self-dual and $T_X=-\widetilde{S}_X^2$ because
$$\langle \widetilde{S}_XY_1, Y_2\rangle=\Re \Tr [X,Y_1]Y_2=\Re \Tr X[Y_1,Y_2]=-\langle Y_1, \widetilde{S}_X Y_2\rangle,$$
$$-\widetilde{S}_X^2 Y=-[X,[X,Y]^*]^*=[X^*,[X,Y]]=T_X Y.$$

Now for any eigenvector $Y\in E_\lambda$, i.e., $T_X Y=\lambda Y$, we claim that $\widetilde{S}_X Y$ is also an eigenvector in $E_\lambda$ which is $\mathbb{C}$-independent (even $\mathbb{C}$-orthogonal) to $Y$. In fact, since $T_X=-\widetilde{S}_X^2$ we have
$$T_X\widetilde{S}_X Y=\widetilde{S}_XT_X Y=\lambda \widetilde{S}_X Y, \quad \|\widetilde{S}_X Y\|^2=\langle T_X Y, Y\rangle=\lambda \|Y\|^2>0,$$
$$\Tr \Big(Y(\widetilde{S}_X Y)^*\Big)=\Tr \Big(Y[X,Y]\Big)=0, \quad \textit{and thus} \quad \langle Y,  \widetilde{S}_X Y\rangle= \langle \oi Y,  \widetilde{S}_X Y\rangle=0,$$
where $\oi=\sqrt{-1}$ here and for the rest of this paper.

For $k\geq1$, suppose that $\Span_{\mathbb{C}}\{Y_i, \widetilde{S}_X Y_i\}_{i=1}^k\subset E_\lambda$ and $Y_{k+1}\in E_\lambda$ is orthogonal to $\Span_{\mathbb{C}}\{Y_i, \widetilde{S}_X Y_i\}_{i=1}^k$. Then it suffices to prove $$\widetilde{S}_X Y_{k+1}\bot \Span_{\mathbb{C}}\{Y_i, \widetilde{S}_X Y_i\}_{i=1}^k.$$
This is easily verified as follows:
$$
\begin{aligned}
&\Tr \Big(Y_i (\widetilde{S}_X Y_{k+1})^*\Big)=- \Tr \Big(Y_{k+1} (\widetilde{S}_X Y_{i})^*\Big)=0,\\
&\Tr \Big(\widetilde{S}_XY_i (\widetilde{S}_X Y_{k+1})^*\Big)=\Tr \Big(Y_{k+1} (T_X Y_{i})^*\Big)=\lambda \Tr \Big(Y_{k+1} (Y_{i})^*\Big)= 0.
\end{aligned}$$
The proof is complete.
\end{proof}

As the pair $(T, S)$ in Lemmas \ref{lem0}-\ref{lemma 2} , we can define a unitary skew-symmetric linear operator $S_X$ on $V=M(n,\mathbb{C})$ such that $T_X=S_X^*S_X=-S_X^2$ as follows.
Taking an orthonormal basis $\{v_i\}_{i=1}^N$ of $V$ such that $v_i$ is an eigenvector of the eigenvalue $\lambda_i(T_X)$,  we define $S_X$ on this basis by
$S_X(v_i):=\widetilde{S}_Xv_i=[X,v_i]^*$ and then extend it linearly to the whole space as
\begin{equation}\label{def-SX}
\begin{aligned}
S_X:  M(n,\mathbb{C}) &\longrightarrow  M(n,\mathbb{C}),\\
 Y=\sum_{i=1}^N y_i v_i &\longmapsto  \sum_{i=1}^Ny_i [X,v_i]^*, \quad \textit{for } y_1,\cdots,y_N\in\mathbb{C}.
\end{aligned}
\end{equation}
In particular, by the proof of Proposition \ref{prop2} we can choose the second half of the eigenvectors $v_i$'s of those positive eigenvalues $\lambda_i(T_X)$ to be the image of $\widetilde{S}_X$, namely, $$v_{i+\widetilde{n_i}}:=\widetilde{S}_Xv_i/\|\widetilde{S}_Xv_i\|=\widetilde{S}_Xv_i/\sqrt{\lambda_i(T_X)},$$
where $2\widetilde{n_i}$ is the even multiplicity of the positive eigenvalues $\lambda_i(T_X)$. As in Lemma \ref{lem0}, suppose that there are $g$ distinct positive eigenvalues $\lambda(T_X)=\{t_1=s_1^2> \cdots> t_g=s_g^2>0\}$  with multiplicities $2n_1, \cdots , 2n_g$, and denote by $r=2\sum_{j=1}^gn_j$ the rank of $T_X$. Under the special basis above, the linear operator $S_X$ can be represented by the real skew-symmetric matrix
$$S_X=\begin{pmatrix}\begin{array}{cc}
O & -s_1I_{n_1}\\
s_1I_{n_1}& O
\end{array}& & & \\
&\ddots&&\\&& \begin{array}{cc}
O & -s_gI_{n_g}\\
s_gI_{n_g}& O
\end{array}&\\ &&&O_{N-r}
\end{pmatrix},$$
while $T_X$ is represented by $T_X=\diag\Big(t_1I_{2n_1},\cdots,t_gI_{2n_g},O_{N-r}\Big)$. One can also reorder the basis in the way $v_{2i}=\widetilde{S}_Xv_{2i-1}/\sqrt{\lambda_{2i-1}(T_X)}$ such that
\begin{equation}\label{SX2}
S_X=\begin{pmatrix}I_{n_1}\otimes\begin{pmatrix}
0 & -s_1\\
s_1& 0
\end{pmatrix}& & & \\
&\ddots&&\\&& I_{n_g}\otimes\begin{pmatrix}
0 & -s_g\\
s_g& 0
\end{pmatrix}&\\ &&&O_{N-r}
\end{pmatrix}.
\end{equation}
Hence, Lemma \ref{lemma 2} is suitable for the pair $(T_X,S_X)$ and will be applied in the proof of the equivalence between Conjecture \ref{conj2} and Conjecture \ref{conj2A}.

We will also need the following notations and useful lemmas.
Let $\Vec$ be the canonical isomorphism from $M(n,\mathbb{C})$ to $\mathbb{C}^{N}$, i.e.
$$\begin{aligned}
\Vec : M(n,\mathbb{C})&\longrightarrow \mathbb{C}^{N},\\
X=(x_{ij})&\longmapsto (x_{11}, \cdots, x_{n1}, x_{12}, \cdots, x_{n2}, \cdots, x_{1n}, \cdots, x_{nn} )^t,
\end{aligned}$$
where $X^t$ is the transpose of $X$. Using Kronecker product of matrices, we have
\begin{lem}\label{lem1vec}\cite{HJ91}
$Vec(AYB)=\(B^t\otimes A\)Vec(Y)$.
\end{lem}
Moreover, $\Vec$ is an isometry since
$\<X, Y\>=\<\Vec(X), \Vec(Y)\>$, and thus we can calculate the eigenvalues of $T_X$ by $$\lambda(T_X)=\lambda\(\Vec\circ T_X\circ (\Vec)^{-1}\).$$

\begin{prop}\label{VecTX}
$\lambda(T_X)=\lambda(K_X^*K_X)
=\lambda(K_1+K_2)$,
where $K_X=I\otimes X\ -X^t\otimes I$ and $K_1=I\otimes X^*X+\overline{X}X^t\otimes I, \ K_2=-X^t\otimes X^*-\overline{X}\otimes X$.
\end{prop}

\begin{proof}
By Lemma \ref{lem1vec}, we have $\Vec\([X, Y]\)=K_X \Vec(Y)$,
where $$K_X=I\otimes X\ -X^t\otimes I$$ is regarded as a linear operator on $\mathbb{C}^N$, or equivalently as a $N\times N$ matrix. It is easily seen that $K_{X^*}=K_X^*$.

Define ${\Phi}_X(Y):=[X, Y]$, then $$\Vec\circ {\Phi}_X\circ (\Vec)^{-1}=K_X, \quad  T_X={\Phi}_{X^*}\circ{\Phi}_X.$$
In particular, we have
$$\Vec\circ T_X\circ (\Vec)^{-1}=K_{X^*}K_X=K_X^*K_X,$$
hence $$\lambda(T_X)=\lambda(K_X^*K_X).$$
By direct calculation, we have $K_X^*K_X=K_1+K_2$,
where $$K_1=I\otimes X^*X+\overline{X}X^t\otimes I, \quad K_2=-\(X^t\otimes X^*+\overline{X}\otimes X\)$$
are Hermitian matrices.
\end{proof}

\begin{cor}\label{cortrTX}
For $X\in M(n,\mathbb{C})$ with $\|X\|=1$, we have
$\Tr T_X=2n-2|\Tr X|^2$. In particular, for  $n=2$, $\lambda_1(T_X)=\lambda_2(T_X)=2-|\Tr X|^2$ and $\lambda_3(T_X)=\lambda_4(T_X)=0$.
\end{cor}
\begin{proof}
It follows immediately from Proposition \ref{VecTX} that
$$
\Tr T_X= \Tr K_1 +\Tr K_2=2n\|X\|^2-2|\Tr X|^2=2n-2|\Tr X|^2.
$$
For $n=2$, the conclusion follows from Proposition \ref{prop2} and the fact that $T_XX=0$ and thus $T_X$ must have a zero eigenvalue.
\end{proof}

To end this section, we prepare two useful lemmas about eigenvalues of Kronecker product and sum of two matrices.
\begin{lem}\label{lem2}\cite{Z11}
Let $A$ and $B$ be $m\times m$ and $n\times n$ complex matrices with eigenvalues
$\lambda_1, \cdots, \lambda_m$ and $\mu_1, \cdots, \mu_n$, respectively.
Then the eigenvalues of $A\otimes B$ are $$\lambda_i\mu_j, 1\leq i\leq m, 1\leq j\leq n,$$
and the eigenvalues of $A\otimes I_n+I_m\otimes B$ are $$\lambda_i+\mu_j, 1\leq i\leq m, 1\leq j\leq n.$$
\end{lem}

\begin{lem}\label{lem3}\cite{Z11}
Let $A, B$ be $n\times n$ Hermitian matrices and $C=A+B$.
If $\alpha_1\geq\cdots\geq\alpha_n$, $\beta_1\geq\cdots\geq\beta_n$, and
$\gamma_1\geq\cdots\geq\gamma_n$ are the eigenvalues of $A, B$, and $C$, respectively.
Then for any sequence $1\leq i_1<\cdots<i_k\leq n$,
$$\sum_{t=1}^k\alpha_{i_t}+\sum_{t=1}^k\beta_{n-k+t}\leq\sum_{t=1}^k\gamma_{i_t}\leq\sum_{t=1}^k\alpha_{i_t}+\sum_{t=1}^k\beta_t.$$
\end{lem}

\section{Some new proofs of the complex BW inequality}\label{sect-BW}
In this section, we will give some new simple proofs of the complex BW inequality by eigenvalue estimates of $T_X$ in (\ref{def-TX}) for $X\in M(n,\mathbb{C})$ with $\|X\|=1$. Each estimate implies $\lambda_1(T_X)\leq2$ and thus the complex BW inequality since for $\|Y\|=1$, $$\|[X,Y]\|^2\leq \max_{\|Y\|=1}\|[X,Y]\|^2=\max_{\|Y\|=1}\langle T_XY,Y\rangle=\lambda_1(T_X)\leq2=2\|X\|^2\|Y\|^2.$$
As a matter of fact, the core of our approach lies in the fact that the multiplicity of positive eigenvalues of $T_X$ is even by Proposition \ref{prop2}.

%The proof of Theorem \ref{first new proof} and Theorem \ref{second new proof} require no extra technical lemmas. In contrast, the proof of theorem \ref{thm1} relies on technical lemma \ref{lem3}, but it gives the characterization of equality holds.

\begin{thm}\label{thm1}
Let $X=A+B\in M(n,\mathbb{C})$ be the canonical decomposition and $\|X\|=1$, where $A$ is Hermitian, $B$ is skew-Hermitian. Then
$$\begin{aligned}
\lambda_1(T_X)
&\leq 2\Big(\max_{i,j}\{-a_ia_j\}+\max_{i,j}\{-b_ib_j\}\Big)
+\Big({\sigma}^2_1(X)+{\sigma}^2_2(X)\Big) \leq2,\\
%&2\|X\|^2\\
\end{aligned}$$
where  ${\sigma}_1(X)\geq\cdots\geq{\sigma}_n(X)$ are singular values of $X$ and   $\lambda(A)=\{a_1, \cdots, a_n\}$, $\lambda(B)=\{b_1\mathbf{i}, \cdots,
b_n\mathbf{i}\}$
are eigenvalues of $A$,
$B$ respectively.
\end{thm}

\begin{proof}
Let ${\sigma}_1(X)\geq\cdots\geq{\sigma}_n(X)$ be singular values of $X$,
then $$\lambda\(X^*X\)=\lambda(\overline{X}X^t)=\{{\sigma}_1^2(X), \cdots, {\sigma}_n^2(X)\}.$$
Hence for $K_1=I\otimes X^*X+\overline{X}X^t\otimes I$ in Proposition \ref{VecTX}, we have by Lemma \ref{lem2}
\begin{equation}\label{eigenvK1}
\lambda(K_1)=\{{\sigma}_i^2(X)+{\sigma}_j^2(X): 1\leq i, j\leq n\}.
\end{equation}
In particular, $\lambda_2(K_1)={\sigma}^2_1(X)+{\sigma}^2_2(X)$.
Let $X=A+B$, where $A$ is Hermitian, $B$ is skew-Hermitian.
Thus for $K_2$ in Proposition \ref{VecTX}, we have $$K_2=-X^t\otimes X^*-\overline{X}\otimes X=2\(B^t\otimes B-A^t\otimes A\).$$
Then by Lemma \ref{lem2},
$$\lambda(-A^t\otimes A)=\{-a_ia_j: 1\leq i, j\leq n\}, $$
$$\lambda(B^t\otimes B)=\{-b_ib_j: 1\leq i, j\leq n\} ,$$
where  $\lambda(A)=\{a_1, \cdots, a_n\}$, $a_1\geq\cdots\geq a_n$;
$\lambda(B)=\{b_1\mathbf{i}, \cdots, b_n\mathbf{i}\}$,
$b_1\geq\cdots\geq b_n$.
Therefore
\begin{equation}\label{BWA}
\begin{aligned}
\lambda_1(-A^t\otimes A)&=\max_{i,j}\{-a_ia_j\}=\max\{\max_{i\neq j}\{-a_ia_j\}, \max_{i}\{-a_i^2\}\}\\
&\leq\max\{\max_{i\neq j}\{-a_ia_j\}, 0\}\leq\max_{i\neq j}\{|a_ia_j|\}\\
&\leq\frac12\max_{i\neq j}\{a_i^2+a_j^2\}\leq\frac12\|A\|^2,
\end{aligned}
\end{equation}
Similarly
\begin{equation}\label{BWB}
\lambda_1(B^t\otimes B)=\max_{i,j}\{-b_ib_j\}\leq\frac12\max_{i\neq j}\{b_i^2+b_j^2\}\leq\frac12\|B\|^2.
\end{equation}
Since $B^t\otimes B$ and $-A^t\otimes A$ are Hermitian, by Lemma \ref{lem3}, we have
\begin{equation}\label{BWK2}
\begin{aligned}
\lambda_1(K_2)&\leq2\(\lambda_1(B^t\otimes B)+\lambda_1(-A^t\otimes A)\)=2\Big(\max_{i,j}\{-a_ia_j\}+\max_{i,j}\{-b_ib_j\}\Big)\\
&\leq\|A\|^2+\|B\|^2=\|X\|^2=1.
\end{aligned}
\end{equation}
Moreover, for $K_X^*K_X=K_1+K_2$ in Proposition \ref{VecTX}, again by Lemma \ref{lem3} we have
$$\lambda_2(K_X^*K_X)\leq \lambda_2(K_1)+\lambda_1(K_2)\leq {\sigma}_1^2(X)+{\sigma}_2^2(X)+\|X\|^2
\leq2\|X\|^2.$$
Finally by Proposition \ref{prop2} and \ref{VecTX}, we have the required estimation  $$\lambda_1(T_X)=\lambda_2(T_X)=\lambda_2(K_X^*K_X)\leq2\|X\|^2=2.$$
The proof is complete.
\end{proof}

For $X\in M(n,\mathbb{C})$ with $\|X\|=1$, we have the following characterization of when $\lambda_1(T_X)$ attains the upper bound $2$.
\begin{thm}\label{equal}
 $\lambda_1(T_X)=2$ if and only if $X=U\diag(X_0, O_{n-2})U^*$ for some $U\in U(n)$, where $X_0\in M(2, \mathbb{C})$ and $\Tr(X_0)=0$.
\end{thm}
\begin{proof}
%and the sufficiency given by \cite{CVD10}.
We first prove the necessity. All the inequalities in the proof of Theorem \ref{thm1} achieve equality when $\lambda_1(T_X)=2$. Thus by the equality conditions of (\ref{BWA}) and (\ref{BWB}), we have $a_1=-a_n=:a\geq0$, $b_1=-b_n=:b\geq0$, and $a_i=b_i=0$ for $1<i<n$. Therefore,
$$\lambda(-A^t\otimes A)=\{a^2, a^2,  0, \cdots, 0, -a^2, -a^2\},$$
$$\lambda(B^t\otimes B)=\{b^2, b^2, 0, \cdots, 0, -b^2, -b^2\},$$
and there exist $U, V\in U(n)$ such that
$$U^*AU=\diag(a, -a, 0, \cdots, 0),$$
$$V^*BV=\diag(b\mathbf{i}, -b\mathbf{i}, 0, \cdots, 0).$$
Hence
$$\Tr(X)=\Tr(A)+\Tr(B)=0.$$
Because (\ref{BWK2}) achieves equality, the eigenspaces of
$\lambda_1(B^t\otimes B)$ and $\lambda_1(-A^t\otimes A)$ have a nontrivial intersection.
Let $U=(u_1, u_2, \cdots, u_n)$, $V=(v_1, v_2, \cdots, v_n)$,
we have
$$Au_1=au_1,\ Au_2=-au_2,\ Au_j=0,\ 3\leq  j\leq n;$$
$$Bv_1=b\mathbf{i}v_1,\ Bv_2=-b\mathbf{i}v_2,\ Bv_j=0,\ 3\leq  j\leq n.$$
Since $A$ is Hermitian and $B$ is skew-Hermitian, we have
$$A^t\overline{u_1}=a\overline{u_1},\ A^t\overline{u_2}=-a\overline{u_2},\
A^t\overline{u_j}=0,\ 3\leq  j\leq n;$$
$$B^t\overline{v_1}=b\mathbf{i}\overline{v_1},\ B^t\overline{v_2}=-b\mathbf{i}\overline{v_2},\
B^t\overline{v_j}=0,\ 3\leq  j\leq n.$$
By the property of Kronecker product,
the eigenspace of $\lambda_1(-A^t\otimes A)$ is
$\Span_{\mathbb{C}}\{\overline{u_1}\otimes u_2, \overline{u_2}\otimes u_1\}$;
the eigenspace of $\lambda_1(B^t\otimes B)$ is
$\Span_{\mathbb{C}}\{\overline{v_1}\otimes v_2, \overline{v_2}\otimes v_1\}.$
Therefore, there exist $k_1, k_2, l_1, l_2\in\mathbb{C}$ and $|k_1|^2+|k_2|^2=|l_1|^2+|l_2|^2\neq0$
such that
\begin{equation}\label{intersect}
k_1\overline{u_1}\otimes u_2+k_2\overline{u_2}\otimes u_1
=l_1\overline{v_1}\otimes v_2+l_2\overline{v_2}\otimes v_1.
\end{equation}
Recall that $U, V\in U(n)$, so we have
$$\overline{k_2}u_2=\overline{l_1(u_1^*v_2)}v_1+\overline{l_2(u_1^*v_1)}v_2, $$
by left multiply $I\otimes u_1^*$ and conjugate (\ref{intersect}).
Similarly, $$\overline{k_1}u_1=\overline{l_1(u_2^*v_2)}v_1+\overline{l_2(u_2^*v_1)}v_2,$$
$$\overline{l_1}v_1=\overline{k_1(v_2^*u_2)}u_1+\overline{k_2(v_2^*u_1)}u_2,$$
$$\overline{l_2}v_2=\overline{k_1(v_1^*u_2)}u_1+\overline{k_2(v_1^*u_1)}u_2.$$
There are two cases to discuss:
\begin{itemize}
  \item If $k_1k_2\neq0$,
  it is easy to see that $\Span_{\mathbb{C}}\{u_1, u_2\}=\Span_{\mathbb{C}}\{v_1, v_2\}$.
  \item If one of $k_1, k_2$ is zero, we can assume without loss of generality that $k_1\neq0$ and $k_2=0$.
  Then we claim that one of $l_1, l_2$ is zero, otherwise
  $$\overline{l_1}v_1=\overline{k_1(v_2^*u_2)}u_1,\ \overline{l_2}v_2=\overline{k_1(v_1^*u_2)}u_1$$
  will lead to a contradiction. So we can also assume without loss of generality that $l_1\neq0, l_2=0$, thus
  $$k_1\overline{u_1}\otimes u_2=l_1\overline{v_1}\otimes v_2.$$
  Since $U, V\in U(n)$, we have $|k_1/l_1|=1$ and \\
  $1=(v_1^t\otimes v_2^*)(\overline{v_1}\otimes v_2)
  =(k_1/l_1)(v_1^t\otimes v_2^*)(\overline{u_1}\otimes u_2)
  =(k_1/l_1)(v_1^t\overline{u_1}\otimes v_2^*u_2).$ \\
  The equality condition of Cauchy-Schwartz inequality implies that
  $u_1, v_1$ are linear dependent and $u_2, v_2$ are linear dependent.
\end{itemize}
In both cases, we have $\Span_{\mathbb{C}}\{u_1, u_2\}=\Span_{\mathbb{C}}\{v_1, v_2\}$. Therefore
$$U^*XU=U^*AU+U^*BU=\diag(a, -a, 0, \cdots, 0)+\diag(B_0, O_{n-2}),$$
where $B_0\in M(2, \mathbb{C})$. Setting $X_0:=\diag(a, -a)+B_0$, we have the necessity.

To prove the sufficiency, since $X_0\in M(2, \mathbb{C})$ and $\Tr X_0=0$, it follows from Lemma \ref{lem-inv} and Corollary \ref{cortrTX} that
$$\lambda_1(T_{X})=\lambda_1(T_{\diag(X_0, O_{n-2})})=\lambda_1(T_{X_0})=2-{\left|\Tr(X_0) \right| }^2=2.$$

This completes the proof.
\end{proof}

Now we give a new proof of the equality condition for the complex BW inequality.
\begin{defn} \cite{BW08}
A pair $(X, Y)$ of $M(n,\mathbb{C})$ is said to be maximal if $X\neq O$, $Y\neq O$ and $\|XY-YX\|^2=2\|X\|^2\|Y\|^2$ is satisfied.
\end{defn}

\begin{cor}\label{corlBWequality}
Let $X,Y\in M(n,\mathbb{C})$ be nonzero matrices. Then $(X, Y)$ is maximal if and only if there exists a unitary matrix $U\in U(n)$  such that
$$X=U\diag(X_0, 0)U^* \quad \textit{and}\quad Y=U\diag(Y_0, 0)U^*$$
with a maximal pair $(X_0, Y_0)$ in $M(2, \mathbb{C})$, i.e., $X_0\bot_{\mathbb{C}} Y_0$ and $\Tr X_0=\Tr Y_0=0$.
\end{cor}
\begin{proof}
Without loss of generality, we assume $\|X\|=\|Y\|=1$. If $(X, Y)$ is maximal, by definition, we have
$$\langle T_XY,Y\rangle=\langle T_YX,X\rangle=\|[X,Y]\|^2=2.$$
Thus $\lambda_1(T_X)=\lambda_1(T_Y)=2$ and hence by Theorem \ref{equal}, there exist unitary matrices $U_1,U_2\in U(n)$  such that
$$X=U_1\diag(X_0, 0)U_1^* \quad \textit{and}\quad Y=U_2\diag(\widetilde{Y_0}, 0)U_2^*$$
with $\Tr X=\Tr Y=0$. Since $Y$ is an eigenvector of the maximal eigenvalue $\lambda_1(T_X)=2$ and $X$ is an eigenvector of the zero eigenvalue of $T_X$, we know immediately $X\bot_{\mathbb{C}} Y$. Moreover, by (\ref{unitary}) and Lemma \ref{lem-inv} we know $U_1^*YU_1$ is an eigenvector of the maximal eigenvalue $\lambda_1(T_{U_1^*XU_1})=\lambda_1(T_{X_0})=2$, which implies $U_1^*YU_1=\diag(Y_0, 0)$ for some $Y_0\in M(2, \mathbb{C})$. This completes the proof of the necessity.

The sufficiency can be verified by direct computation (cf. \cite{BW08}).
\end{proof}

Let $\|X\|_{(2),2}$ be the (2, 2)-norm defined by
$$\|X\|_{(2),2}=\sqrt{{\sigma}^2_1(X)+{\sigma}^2_2(X)}.$$ For $X\in M(n,\mathbb{R})$, Lu \cite{Lu12} has already proved
$$\lambda_1(T_X)\leq2\|X\|^2_{(2),2}.$$
In fact, we can show this inequality holds also for $X\in M(n,\mathbb{C})$.
\begin{thm}\label{first new proof}
For $X\in M(n,\mathbb{C})$ with $\|X\|=1$, we have  $\lambda_1(T_X)\leq2\|X\|^2_{(2),2}\leq2.$
\end{thm}
\begin{proof}
For $Y\in M(n,\mathbb{C})$, by Proposition \ref{VecTX} we have $$\<WW^*\widetilde{v}, \widetilde{v}\>=\<T_X  Y, Y\>,$$
where
$$W=
\begin{pmatrix}
    I\otimes X^*   &    O    \\
    -\overline{X}\otimes I &    O   \\
\end{pmatrix}_{2N\times 2N},\
\widetilde{v}=
\begin{pmatrix}
     \Vec Y\\
     \Vec Y\\
\end{pmatrix}.$$
Noticing that $$W^*W=I\otimes XX^*+X^t\overline{X}\otimes I,$$ we have by Proposition \ref{prop2} and Lemma \ref{lem2} that $$\lambda_1(T_X)=\lambda_2(T_X)\leq
2\lambda_2(WW^*)=2\lambda_2(W^*W)=2(\sigma_1^2(X)+\sigma_2^2(X))=2\|X\|^2_{(2),2}.$$
This completes the proof.
\end{proof}

Denote the upper bound in Theorem \ref{thm1} by $$C_X:=2(\max_{i,j}\{-a_ia_j\}+\max_{i,j}\{-b_ib_j\})+\|X\|^2_{(2),2}.$$
It worths remarking that
$C_X\leq2\|X\|^2_{(2),2}$ if $\rank(X)\leq2$. In general,  $C_X$ is not necessarily less than $2\|X\|^2_{(2),2}$.
However, we are able to obtain
$C_X\leq3\|X\|^2_{(2),2}$,
 since $\{{\left|a_j-\mathbf{i}b_{n-j+1}\right| }^2\}_{j=1}^n$
is majorized by $\{{\sigma}^2_j(X)\}_{j=1}^n$
due to Ando-Bhatia \cite{ATRB87}. Therefore these two upper bounds are strictly different. Combining Theorems \ref{thm1} and \ref{first new proof}, we have the following estimate.
\begin{cor}\label{corlBWNEW}
For $X\in M(n,\mathbb{C})$ with $\|X\|=1$, we have
$$\lambda_1(T_X) \leq
\min\{C_X,2\|X\|^2_{(2),2}\}\leq2.$$
\end{cor}

Furthermore, our approach can be used to estimate all eigenvalues of $T_X$ by that of $K_1$ in Proposition \ref{VecTX}.
Recall that the set of eigenvalues of $K_1$ is given in (\ref{eigenvK1}):
$$\lambda(K_1)=\{{\sigma}_i^2(X)+{\sigma}_j^2(X): 1\leq i, j\leq n\}.$$
\begin{thm}\label{second new proof}
For $X\in M(n,\mathbb{C})$ with $\|X\|=1$, we have $\lambda_i(T_X)\leq2\lambda_i(K_1)$ for all $i$.
\end{thm}
\begin{proof}
Recall that $K_1=I\otimes X^*X+\overline{X}X^t\otimes I, \ K_2=-\(X^t\otimes X^*+\overline{X}\otimes X\)$, and
$$\Vec\circ T_X\circ (\Vec)^{-1}=K_1+K_2.$$
Let $\widehat{K}_X:=I\otimes X+X^t\otimes I$. Then we observe
$$2K_1-\Vec\circ T_X\circ (\Vec)^{-1}=K_1-K_2=\widehat{K}_X^*\widehat{K}_X\geq0,$$
which implies  $$\lambda_i(T_X)\leq2\lambda_i(K_1), \quad \textit{for all } i.$$
The proof is complete.
\end{proof}

In particular, Theorem \ref{second new proof} implies Theorem \ref{first new proof} since $$\lambda_1(T_X)=\lambda_2(T_X)\leq2\lambda_2(K_1)=2(\sigma_1^2+\sigma_2^2)=2\|X\|^2_{(2),2}.$$

%On this basis, we propose the general question \ref{quesBWNEW} before concluding this section. Though the question \ref{quesBWNEW} is not true for every complex square matrix $X$ and
%$\lambda_3(T_{A+B})$ is not always less than
%$\lambda_3(T_{A+\mathbf{i}B})$ for every real square matrix $X$.
%
%\begin{ques}\label{quesBWNEW}
%Whether or not $\lambda_1(T_{A+B})\leq\lambda_1(T_{A+\mathbf{i}B})$
%for every real square matrix $X$ with $X=A+B$, where $A$ is real symmetric and $B$ is real skew-symmetric.
%\end{ques}

\section{Equivalence of the conjectures with the LW conjecture}\label{sect-eq}
In this section, we prove the equivalence between Conjectures \ref{conj2A}-\ref{conj2C} and Conjecture \ref{conj2}, i.e., Theorem \ref{equiv conj} in the complex version.
This theorem will be divided into the following propositions.

\begin{prop} \label{prop equ2 5}
Conjecture \ref{conj2} is equivalent to Conjecture \ref{conj2A}.
\end{prop}
\begin{proof}
Assume Conjecture \ref{conj2} is true at first.
Setting $B=X$ and $B_\alpha$ be a unit eigenvector of $\lambda_{2\alpha-3}(T_X)$ for $\alpha=2,\cdots,m$, by the last expression of $S_X$ in (\ref{def-SX}, \ref{SX2}) we know
$S_XB_\alpha=[B,B_\alpha]^*$ is exactly an eigenvector of $\lambda_{2\alpha-2}(T_X)$. Therefore the conditions (i,ii) of  Conjecture \ref{conj2} are satisfied and thus
we have the inequality  (\ref{ineq of conj2}). Then the inequality (\ref{ineq of conj2A}) of Conjecture \ref{conj2A} for $k=m-1$ follows by Proposition \ref{prop2} and the following
$$\begin{aligned}
\sum_{i=1}^{2k}\lambda_{i}(T_X)&=2\sum_{\alpha=2}^{m}\lambda_{2\alpha-3}(T_X)=2\sum_{\alpha=2}^{m}\langle T_BB_\alpha, B_\alpha\rangle=2\sum_{\alpha=2}^{m}\|[B,B_\alpha]\|^2\\
&\leq 2\(\max \limits_{2\leq \alpha\leq m}\|B_\alpha\|^2+\sum^m_{\alpha=2}\|B_\alpha\|^2\)\|B\|^2=2m=2(k+1).
\end{aligned}$$

Now we assume Conjecture \ref{conj2A} is true.
Without loss of generality, we assume $1=\|B\|\geq\|B_2\|\geq\cdots\geq\|B_m\|>0$.
Using summation by parts, we can write
$$\begin{aligned}
\sum^m_{\alpha=2}\|\[B,B_\alpha\]\|^2
&=\sum^m_{\alpha=2}\<T_BB_\alpha, B_\alpha\>=\sum^m_{\alpha=2}\<T_B\frac{B_\alpha}{\|B_\alpha\|}, \frac{B_\alpha}{\|B_\alpha\|}\>\|B_\alpha\|^2\\
&=\sum^m_{\beta=2}\(\|B_\beta\|^2-\|B_{\beta+1}\|^2\)\sum^\beta_{\alpha=2}\<T_B\frac{B_\alpha}{\|B_\alpha\|}, \frac{B_\alpha}{\|B_\alpha\|}\>,
\end{aligned}$$
where $B_{m+1}=0$. Setting $X=B$, the conditions (i,ii) of  Conjecture \ref{conj2} show that the subspace $W:=\Span_{\mathbb{C}}\{B_\alpha\}_{\alpha=2}^m$ is isotropic about $S_X$, i.e., $S_X(W)\bot_\mathbb{C} W$. Then by the formula above, Lemma \ref{lemma 2} and the inequality (\ref{ineq of conj2A}) of  Conjecture \ref{conj2A}, we have
$$\begin{aligned}
\sum^m_{\alpha=2}\|\[B,B_\alpha\]\|^2&\leq
\sum^m_{\beta=2}\(\|B_\beta\|^2-\|B_{\beta+1}\|^2\)\sum^\beta_{\alpha=2}\lambda_{2\alpha-3}(T_X)\\
&\leq\sum^m_{\beta=2}\(\|B_\beta\|^2-\|B_{\beta+1}\|^2\)\beta\\
&=\|B_2\|^2+\sum^m_{\alpha=2}\|B_\alpha\|^2,
\end{aligned}$$
which is the inequality (\ref{ineq of conj2}) of Conjecture \ref{conj2}.

The proof is complete.
\end{proof}

\begin{prop}\label{eq-45}
Conjecture \ref{conj2A} is equivalent to Conjecture \ref{conj2B}.
\end{prop}
\begin{proof}
Obviously Conjecture \ref{conj2B} implies Conjecture \ref{conj2A} by definition. Suppose Conjecture \ref{conj2A} be true. To prove Conjecture \ref{conj2B}, we only need to prove the following four parts:
\begin{enumerate}[(i)]
\item $\lambda_1(T_X) \leq 2;$
\item $\sum_{i=1}^{2k}\lambda_{i}(T_X)\leq2k+2;$
\item $\sum_{i=1}^{2k-1}\lambda_{i}(T_X)\leq2k+1;$
\item $\sum_{i=1}^{N}\lambda_{i}(T_X)=2n-2|\Tr X|^2\leq2n$,
\end{enumerate}
where (i) and (iv) are ensured by the complex BW inequality (e.g., Theorem \ref{thm1}) and Corollary \ref{cortrTX}, and (ii) is assumed by Conjecture \ref{conj2A}, respectively. We are left to show the inequality (iii). We prove it by contradiction in the following.

 Assume that there is a positive number $m\geq2$ such that
$$\sum_{i=1}^{2m-1}\lambda_{i}(T_X)>2m+1.$$
Then
$$2m+1<\sum_{i=1}^{2m-1}\lambda_{i}(T_X)=\lambda_{2m-1}(T_X)+\sum_{i=1}^{2m-2}\lambda_{i}(T_X)\leq \lambda_{2m-1}(T_X)+2m.$$
Thus
$$\lambda_{2m}(T_X)=\lambda_{2m-1}(T_X)>1,$$
and
$$\sum_{i=1}^{2m}\lambda_{i}(T_X)=\lambda_{2m}(T_X)+\sum_{i=1}^{2m-1}\lambda_{i}(T_X)>1+2m+1=2m+2.$$
This leads to the contradiction to (ii) and completes the proof.
\end{proof}

\begin{prop} \label{prop equ5 7}
Conjecture \ref{conj2A} is equivalent to Conjecture \ref{conj2C}.
\end{prop}
\begin{proof}
The proof is similar to that of  Proposition \ref{prop equ2 5}, without using Lemma \ref{lemma 2} now since we have no condition (ii) of Conjecture \ref{conj2}.
For the sake of clearness, we repeat it as follows.

Assume Conjecture \ref{conj2C} is true.
Setting $B=X$ and $B_\alpha$ be a unit eigenvector of $\lambda_{\alpha-1}(T_X)$ for $\alpha=2,\cdots,m$,  we know
$B_\alpha$'s are $\mathbb{C}$-orthogonal and  therefore we have the inequality of  Conjecture \ref{conj2C}. Then the inequality (\ref{ineq of conj2A}) of Conjecture \ref{conj2A} for $m=2k+1$ follows by
$$\begin{aligned}
\sum_{i=1}^{m-1}\lambda_{i}(T_X)&=\sum_{\alpha=2}^{m}\lambda_{\alpha-1}(T_X)=\sum_{\alpha=2}^{m}\langle T_BB_\alpha, B_\alpha\rangle=\sum_{\alpha=2}^{m}\|[B,B_\alpha]\|^2\\
&\leq \(2\max \limits_{2\leq \alpha\leq m}\|B_\alpha\|^2+\sum^m_{\alpha=2}\|B_\alpha\|^2\)\|B\|^2=m+1.
\end{aligned}$$

Now we assume Conjecture \ref{conj2A} is true and hence Conjecture \ref{conj2B} is true by Proposition \ref{eq-45}. In particular, we have
$$\sum_{i=1}^m\lambda_i(T_X)\leq m+2, \quad \textit{for any } m.$$
Without loss of generality, we assume $1=\|B\|\geq\|B_2\|\geq\cdots\geq\|B_m\|>0$ and set $B_{m+1}=0$.
Then using summation by parts, we have
$$\begin{aligned}
\sum^m_{\alpha=2}\|\[B,B_\alpha\]\|^2
&=\sum^m_{\alpha=2}\<T_BB_\alpha, B_\alpha\>=\sum^m_{\alpha=2}\<T_B\frac{B_\alpha}{\|B_\alpha\|}, \frac{B_\alpha}{\|B_\alpha\|}\>\|B_\alpha\|^2\\
&=\sum^m_{\beta=2}\(\|B_\beta\|^2-\|B_{\beta+1}\|^2\)\sum^\beta_{\alpha=2}\<T_B\frac{B_\alpha}{\|B_\alpha\|}, \frac{B_\alpha}{\|B_\alpha\|}\>,\\
&\leq \sum^m_{\beta=2}\(\|B_\beta\|^2-\|B_{\beta+1}\|^2\)\sum^{\beta-1}_{\alpha=1}\lambda_{\alpha}(T_X)\\
&\leq\sum^m_{\beta=2}\(\|B_\beta\|^2-\|B_{\beta+1}\|^2\)(\beta+1)\\
&=2\|B_2\|^2+\sum^m_{\alpha=2}\|B_\alpha\|^2,
\end{aligned}$$
which is the inequality of Conjecture \ref{conj2C}.

The proof is complete.
\end{proof}

\begin{prop}  \cite{LW16}
The LW Conjecture \ref{conj2} implies Conjectures \ref{conj1} and \ref{conj3}.
\end{prop}
\begin{proof}
Conjecture \ref{conj3} is trivially implied by Conjecture \ref{conj2A} and thus by Conjecture \ref{conj2}.

As for Conjecture \ref{conj1}, for the sake of completeness, we copy the proof of the real version from \cite{LW16} for our complex version now.

We first observe that the inequality (\ref{csymDDVV}) is invariant under the transformations
$$\begin{aligned}
M(n,\mathbb{C}) &\longrightarrow  M(n,\mathbb{C}),\\
A_\alpha &\longmapsto  QA_\alpha Q^*, \\
A_\alpha &\longmapsto  \sum_{\beta=1}^mp_{\alpha \beta}A_\beta,
\end{aligned}
$$
for all unitary $n\times n$ matrices $Q$ and $m\times m$ matrices $P=(p_{\alpha \beta})$.

Now, let $a>0$ be the largest positive real number such that
$$
(\sum_{\alpha=1}^m||A_\alpha||^2)^2\geq 2a(\sum_{\alpha<\beta}||[A_\alpha,A_\beta]||^2)
$$
for all matrices $A_\alpha$'s satisfying the condition of Conjecture \ref{conj1}.

Since $a$ is maximal, by the invariance we can find matrices $A_1,\cdots,A_m$ such that
\begin{equation}\label{conj-1}
(\sum_{\alpha=1}^m||A_\alpha||^2)^2= 2a(\sum_{\alpha<\beta}||[A_\alpha,A_\beta]||^2)
\end{equation}
with the following additional properties:
\begin{enumerate}
\item $\Tr A_\alpha A_\beta^*=0$ for any $\alpha\neq \beta$;
\item $\Tr A_\alpha\[A_\gamma, A_\beta\] =0$ for any $1\leq \alpha,\beta,\gamma\leq m$;
\item $0\neq ||A_1||\geq ||A_2||\geq\cdots \geq ||A_m||$.
\end{enumerate}

We let $t^2=||A_1||^2$ and let $A'=A_1/|t|$. Then ~\eqref{conj-1} becomes a quadratic expression in terms of $t^2$:
\begin{align*}&
t^4-2t^2\Big(a\sum_{1<\alpha}||[A',A_\alpha]||^2-\sum_{1<\alpha}||A_\alpha||^2\Big)
+\Big(\sum_{\alpha=2}^m||A_\alpha||^2\Big)^2\\&\qquad- 2a\Big(\sum_{1<\alpha<\beta}||[A_\alpha,A_\beta]||^2\Big)=0.
\end{align*}
Since the left-hand side of the above is nonnegative for all $t^2\geq0$ and is zero for $t^2=||A_1||^2$, we have
$$
a\sum_{1<\alpha}||[A',A_\alpha]||^2-\sum_{1<\alpha}||A_\alpha||^2>0,
$$
and
$$
||A_1||^2=a\sum_{1<\alpha}||[A',A_\alpha]||^2-\sum_{1<\alpha}||A_\alpha||^2.
$$
By Conjecture \ref{conj2}, we have
$$
\sum_{1<\alpha}||[A',A_\alpha]||^2\leq \sum_{\alpha=2}^m||A_\alpha||^2+||A_2||^2\leq \sum_{\alpha=1}^m||A_\alpha||^2,
$$
which proves that $a\geq 1$ and this completes the proof.
\end{proof}

%\begin{cor}\label{DDVVcor}
%DDVV conjecture is true if
%$$\sum_{t=1}^{2k}\lambda_{i}(T_X)\leq2k
%+2.$$
%for every positive integer $k$ and real square symmetric  matrix $X$ with $\|X\|=1$.
%\end{cor}

\section{Partial results on the complex LW Conjecture}\label{sect-LW}
In this section, we prove the complex LW Conjecture separately for those special cases (Theorem \ref{thm-special-cases}), and for general cases, we give some non-sharp upper bounds for the inequalities of Conjectures \ref{conj3} and \ref{conj2D} (Theorems \ref{thm-conj3-weak} and \ref{thm-conj7-weak}).

Firstly we prove the complex version of Conjecture \ref{conj3} for the first special case of Theorem \ref{thm-special-cases}. We remind that Conjecture \ref{conj3} is also the first step of the complex LW Conjecture \ref{conj2D} after the solution of the BW inequality (i.e., $\lambda_1(T_X)\leq2$).
\begin{thm}\label{normallemma}
Conjecture \ref{conj3} is true when $X$ is a normal matrix. % with $\|X\|=1$.
\end{thm}
\begin{proof}
Since $X$ is a normal matrix, there exists a unitary matrix $U$
such that $$U^*XU=\diag(x_1, \cdots, x_n), \quad \textit{for some } x_1, \cdots, x_n\in \mathbb{C} \textit{ with } \sum_i|x_i|^2=1.$$
Direct calculations show that for any $1\leq i, j\leq n$,
$$T_{U^*XU}(E_{ij}) = |x_i-x_j|^2E_{ij},$$
where $E_{ij}\in M(n,\mathbb{C})$ is the standard basis matrix with the $(i,j)$-element being $1$ and the others being $0$.
Then by the identity (\ref{unitary}):
$$T_{U^*XU}(U^*YU) = U^*T_X(Y)U,$$
we have
$$T_X(UE_{ij}U^*) = UT_{U^*XU}(E_{ij})U^* = |x_i-x_j|^2UE_{ij}U^*. $$
It follows that
$$\lambda(T_X)=\left\{|x_i-x_j|^2: 1\leq i, j\leq n\right\}=\{\lambda_1\geq\cdots\geq \lambda_{n^2}\}.$$
Suppose $\lambda_1=\lambda_2=|x_a-x_b|^2$, $\lambda_3=\lambda_4=|x_c-x_d|^2$,
where $1\leq a, b, c, d\leq n$.
There are two cases need to be discussed:
\begin{itemize}
  \item If $a, b, c, d$ are four different integers, then
  $$\lambda_1+\lambda_3=|x_a-x_b|^2+|x_c-x_d|^2\leq 2(|x_a|^2+|x_b|^2+|x_c|^2+|x_d|^2)\leq 2.$$
  \item If one of $a, b$ is equal to one of $c, d$, we can assume $a=c$, $b\neq d$.
  Then $$\begin{aligned}
  \lambda_1+\lambda_3&=|x_a-x_b|^2+|x_a-x_d|^2\\
  &=|x_a|^2-x_a\overline{x_b}-\overline{x_a}x_b+|x_b|^2+|x_a|^2-x_a\overline{x_d}-\overline{x_a}x_d+|x_d|^2\\
  &\leq 2|x_a|^2+2|x_a|\(|x_b|+|x_d|\)+|x_b|^2+|x_d|^2\\
  &\leq 3|x_a|^2+\(|x_b|+|x_d|\)^2+|x_b|^2+|x_d|^2\\
  &\leq 3(|x_a|^2+|x_b|^2+|x_d|^2)\leq 3.
  \end{aligned}$$
  The equality holds if and only if $|x_a|=\frac{\sqrt{6}}3$, $|x_b|=|x_d|=\frac{\sqrt{6}}6$, other $x_e=0$
  and $x_a, x_b, x_d$ are co-linear in the complex plane.
\end{itemize}
The proof is complete.
\end{proof}

For more general cases, we need Lu's lemma in the complex version:
%\begin{lem}\label{Lulemma}\cite{Lu11}
%Suppose $\eta_1,\cdots,\eta_n$ are complex numbers and
%$$
%\eta_1+\cdots+\eta_n=0,\quad
%\eta_1^2+\cdots+\eta_n^2=1.
%$$
%Let $r_{ij}\geq 0$ be nonnegative numbers for $i<j$. Then we have
%\begin{equation}\label{use}
%\sum_{i<j}(\eta_i-\eta_j)^2 r_{ij}\leq\sum_{i<j}r_{ij}+\max\limits_{i<j} (r_{ij}).
%\end{equation}
%%If $\eta_1\geq\cdots\geq\eta_n$, and $r_{ij}$ are not simultaneously zero, then the equality in~\eqref{use} holds in one of the following three cases:
%%\begin{enumerate}
%%\item $r_{ij}=0$ unless $(i,j)=(1,n)$, $(\eta_1,\cdots,\eta_n)=(1/\sqrt 2,0,\cdots,0, -1/\sqrt 2)$;
%%\item $r_{ij}=0$ if $2\leq i<j$, $r_{12}=\cdots=r_{1n}\neq 0$, and\newline  $(\eta_1,\cdots,\eta_n)=(\sqrt{(n-1)/n},-1/\sqrt{n(n-1)},\cdots,-1/\sqrt{n(n-1)})$;
%%\item $r_{ij}=0$ if $i<j<n$, $r_{1n}=\cdots=r_{(n-1)n}\neq 0$, and\newline
%% $(\eta_1,\cdots,\eta_n)=(1/\sqrt{n(n-1)},\cdots,1/\sqrt{n(n-1)}, -\sqrt{(n-1)/n})$;
%%\end{enumerate}
%\end{lem}
%
%It is easy to get corollary \ref{Lucorcom1} and corollary \ref{Lucorcom2} by lemma \ref{Lulemma}.
\begin{lem}\label{Lucorcom1}\cite{Lu11}
Suppose $\eta_1,\cdots,\eta_n$ are complex numbers and
$$
\eta_1+\cdots+\eta_n=0,\quad
|\eta_1|^2+\cdots+|\eta_n|^2=1.
$$
Let $r_{ij}\geq 0$ be nonnegative numbers for $i<j$. Then we have
\begin{equation}
\sum_{i<j}|\eta_i-\eta_j|^2 r_{ij}\leq\sum_{i<j}r_{ij}+\max\limits_{i<j} (r_{ij}).
\end{equation}
\end{lem}

\begin{cor}\label{Lucorcom2}
The complex LW Conjecture \ref{conj2D} is true when $X$ is a normal matrix.
\end{cor}
\begin{proof}
Let $X$ be a normal matrix and $r_{ij}\in \{0,1\}$,
then it follows from the proof of Theorem \ref{normallemma} that
$$\lambda(T_X)=\left\{|\eta_i-\eta_j|^2: 1\leq i, j\leq n\right\},$$
where $\lambda(X)=\{\eta_1,\eta_2,\cdots,\eta_n\}$. Thus Corollary \ref{Lucorcom1} applies to tell us
$$\sum_{\alpha=1}^k\lambda_{2\alpha-1}\leq k+1,$$
where  $\lambda_{2\alpha-1}$ equals some $|\eta_i-\eta_j|^2$ and $r_{ij}=1$ for $k$ pairs of $(i<j)$.

This completes the proof.
\end{proof}

\begin{cor}
Let $B_1,\cdots,B_m\in M(n,\mathbb{C})$ be Hermitian metrices. Assume that
\begin{equation}\label{trace}
\Tr \Big(B_\alpha [B_\gamma, B_\beta]\Big)=0
\end{equation}
 for any $1\leq \alpha,\beta,\gamma\leq m$, we have
\begin{equation*}
\sum^m_{\alpha,\beta=1}\|\[B_\alpha,B_\beta\]\|^2\leq  \(\sum^m_{\alpha=1}\|B_\alpha\|^2\)^2.
\end{equation*}
\end{cor}
\begin{proof}
As Hermitian matrices are normal matrices, by Corollary \ref{Lucorcom2} above, the complex LW Conjecture \ref{conj2D} holds for this case. This in turn by Theorem \ref{equiv conj} implies the complex version of Conjecture \ref{conj1}.
%which exactly gives the sharper DDVV-type inequality for these special Hermitian matrices. We remind that for general Hermitian matrices the optimal constant $c=\frac{4}{3}$ is bigger than $1$ here (cf. Section \ref{sect-intr},  \cite{GXYZ17}, \cite{GLZ18}).
\end{proof}
\begin{rem} When $B_1,\cdots, B_m$ are real symmetric matrices, ~\eqref{trace} is valid for all $\alpha,\beta,\gamma$. Thus the corollary generalizes the DDVV inequality and is sharp under the trace condition~\eqref{trace}. We remind that for general Hermitian matrices the optimal constant $c=\frac{4}{3}$ is bigger than $1$ here (cf. Section \ref{sect-intr},  \cite{GXYZ17}, \cite{GLZ18}).
\end{rem}

Next we prove Conjecture \ref{conj3} for the second special case $\rank(X)=1$.
We will need the following lemma.

\begin{lem}\label{lemThompson}\cite{BR97}
Let $M\in M(n,\mathbb{C})$ be a complex matrix. Then
$$\lambda_i(\frac{M^*+M}{2})\leq \sigma_i(M), \quad i=1, \cdots,n,$$
where $\lambda_i$ and $\sigma_i$ are eigenvalues and singular values, respectively.
\end{lem}

\begin{thm}\label{thm3}
The complex LW Conjecture \ref{conj2D} is true when $\rank(X)=1$.
\end{thm}
\begin{proof}
Recall Proposition \ref{VecTX} that we have $K_X^*K_X=K_1+K_2,$
where $K_1=I\otimes X^*X+\overline{X}X^t\otimes I, \ K_2=-\(X^t\otimes X^*+\overline{X}\otimes X\).$
Denote $K_3=-X^t\otimes X^*$, then $K_2=K_3^*+K_3$ and by Lemma \ref{lemThompson},
$$\lambda_i(K_2)=2\lambda_i(\frac{K_3^*+K_3}{2})\leq 2\sigma_i(K_3), \quad i=1, \cdots,n.$$
Let ${\sigma}_1(X)\geq\cdots\geq{\sigma}_n(X)$ be singular values of $X$, then by Lemma \ref{lem2},
$$\sigma(K_3)=\{{\sigma}_i(X){\sigma}_j(X): 1\leq i, j\leq n\}.$$
In particular, now $\rank(X)=1$ implies ${\sigma}_1(X)=1$ and ${\sigma}_i(X)=0$ for $2\leq i\leq n$. Thus we have
$\sigma(K_3)=\{1^1,0^{N-1}\}$ and by (\ref{eigenvK1})
$$\lambda(K_1)=\{{\sigma}_i(X)^2+{\sigma}_j(X)^2: 1\leq i, j\leq n\}=\{2^1,1^{2(n-1)},0^{(n-1)^2}\}.$$

Finally by Propositions \ref{prop2}, \ref{VecTX} and Lemma \ref{lem3}, we have
$$\begin{aligned}
\sum_{i=1}^{k}\lambda_{2i-1}(T_X)=\sum_{i=1}^{k}\lambda_{2i}(T_X)
&\leq\sum_{i=1}^{k}\lambda_i(K_1)+\sum_{i=1}^{k}\lambda_{2i}(K_2)\\
&\leq\sum_{i=1}^{k}\lambda_i(K_1)+\sum_{i=1}^{k}2\sigma_{2i}(K_3)\\
&=\sum_{i=1}^{k}\lambda_i(K_1)\leq k+1,
\end{aligned}$$
which completes the proof.
\end{proof}

Furthermore, we can get the characteristic polynomial of $T_X$ if $\rank(X)=1$.
\begin{prop}\label{prop3}
Let $K_X=I\otimes X\ -X^t\otimes I$. Then the sets of singular values  $$\sigma(K_X)=\sigma\Big(I\otimes \Lambda\ -\(\Lambda\otimes I\)\( Q^t\otimes \ Q^*\)\Big),$$
where $X=Q_1\Lambda Q_2$ is the singular value decomposition of $X$ and $Q=Q_2Q_1$.
\end{prop}
\begin{proof}
Direct calculations show
$$\begin{aligned}
K_X
&=I\otimes X\ -X^t\otimes I\\
&=I\otimes \(Q_1\Lambda Q_2\) -\(Q_2^t\Lambda Q_1^t\)\otimes I\\
&=\( Q_2^t\otimes \ Q_1\)
\[ I\otimes \Lambda\ -\(\Lambda\otimes I\)\( Q^t\otimes \ Q^*\) \]
\( \overline{Q_2}\otimes \ Q_2\).\\
\end{aligned}$$
This completes the proof by the invariance of singular values under congruences.
\end{proof}

\begin{thm}%\label{thm3}
Let $X$ be a complex square matrix of order $n$ $(\geq2)$ with $\|X\|=1$ and $rank(X)=1$, then the characteristic polynomial of $T_X$ is
$$\det(\lambda I-T_X)= \(\lambda-2+\left| \Tr X \right|^2\)^2\(\lambda-1\)^{2n-4}\lambda^{\(n-1\)^2+1}.$$
 %$\lambda(T_X)=\{2-q_{11}^2, 1, 0\}$.
\end{thm}
\begin{proof}
Let $X=Q_1\Lambda Q_2$ be the singular value decomposition and $Q=Q_1^tQ_2^t=:(q_{ij})$.
Proposition \ref{prop3} implies
$\sigma(K_X)=\sigma(\widetilde{K_X})$, where
$\widetilde{K_X}=I\otimes \Lambda\ -\(\Lambda\otimes I\)\( Q\otimes \ \overline{Q}\)$.
By Proposition \ref{VecTX}, we have
$$\lambda(T_X)=\lambda(K_XK_X^*)=\lambda(\widetilde{K_X}\widetilde{K_X}^*),$$  where direct calculations show
$$\widetilde{K_X}\widetilde{K_X}^*=I\otimes \Lambda^2\ +\Lambda^2\otimes I\
- \(Q^*\Lambda\)\otimes \(\Lambda Q^t\)
- \(\Lambda Q\)\otimes \(\overline{Q}\Lambda \).$$
Since $\|X\|$=1 and $\rank(X)=1$, it implies $\Lambda=\diag(1,0,\cdots,0)$. By direct calculations, we have
$I\otimes \Lambda^2\ +\Lambda^2\otimes I\
=\diag(I+\Lambda, \Lambda,\cdots, \Lambda)$ and thus
$$\lambda\(I\otimes I\) -\widetilde{K_X}\widetilde{K_X}^*=
\begin{pmatrix}
 A & B \\
C & D \\
\end{pmatrix},$$
where $$A:=\(\lambda-1\)I-\Lambda+q_{11}\overline{Q}\Lambda+ \overline{q_{11}}\Lambda Q^t,\quad
B:=\(q_{12}\overline{Q}\Lambda,q_{13}\overline{Q}\Lambda,\cdots,q_{1n}\overline{Q}\Lambda\),$$
%$C:=\(\bar{q_{12}}\Lambda Q^t,\bar{q_{13}}\Lambda Q^t,\cdots,\bar{q_{1n}}\Lambda Q^t\)$
$$C:=B^*, \quad D:=\diag(\lambda I-\Lambda, \lambda I-\Lambda,\cdots, \lambda I-\Lambda).$$
Without loss of generality, suppose that the determinant of matrix $D$ is not zero, then
$$\begin{aligned}
\det
\begin{pmatrix}
 A & B \\
C & D \\
\end{pmatrix}
&=\det\left(A-BD^{-1}C\right)\cdot\det D\\
&=\det\left(A-\(1-\left| q_{11}\right|^2\)\overline{Q}\Lambda \widehat{D}\Lambda Q^t\right)\cdot\det D,\\
\end{aligned}$$
where $\widehat{D}=\diag(\frac{1}{\lambda-1},\frac{1}
{\lambda}\cdots,\frac{1}{\lambda})$.
\\Thus
$$A-BD^{-1}C= A-\(1-\left| q_{11} \right| ^2\)\overline{Q}\Lambda \widehat{D} \Lambda Q^t=
\begin{pmatrix}
\widetilde{A} & \widetilde{B} \\
\widetilde{C}  & \widetilde{D} \\
\end{pmatrix},$$
where $\widetilde{A}:=\lambda-2+2\left| q_{11} \right| ^2-\frac{1-\left| q_{11} \right| ^2}{\lambda-1}\left| q_{11} \right| ^2,$
$\widetilde{B}:=(\overline{q_{11}}q_{21}\frac{
\lambda -2+\left| q_{11}\right| ^2}
{\lambda-1},\cdots,\overline{q_{11}}q_{n1}\frac{\lambda -2+\left| q_{11}\right| ^2}{\lambda-1}
)$,
$ \widetilde{C}:=\widetilde{B}^*$,
$\widetilde{D}:=\(\lambda-1\)I-\frac{1-\left| q_{11} \right| ^2}{\lambda-1}u^*u$, $u:=(q_{21},q_{31},\cdots,q_{n1})$.
\\Similarly,
$$\begin{aligned}
\det
\begin{pmatrix}
\widetilde{A} & \widetilde{B} \\
\widetilde{C}  & \widetilde{D} \\
\end{pmatrix}
&=\det\left(\widetilde{A}-\widetilde{C}\widetilde{A}^{-1}\widetilde{B}\right) \cdot\det\widetilde{A}\\
&=\det\left(\(\lambda-1\)I-\frac{1}{\lambda-1+\left| q_{11}\right| ^2}u^*u\right) \cdot\det\widetilde{A}\\
&=\left[\lambda\frac{\lambda-2+\left| q_{11}\right| ^2}{\lambda-1+\left| q_{11}\right| ^2}
\(\lambda-1\)^{n-2}\right]\cdot
\left[\frac{1}{\lambda-1}\(\lambda-1+\left|q_{11}\right|^2\)\(\lambda-2+\left| q_{11}\right| ^2\)\right]\\
&=\(\lambda-2+\left| q_{11} \right|^2\)^2\(\lambda-1\)^{n-3}\lambda.\\
\end{aligned}$$
So we have
$$\begin{aligned}
\det
\begin{pmatrix}
 A & B \\
C & D \\
\end{pmatrix}
&=\det\left(A-BD^{-1}C\right)\cdot\det D\\
&=\det\left(\widetilde{A}-\widetilde{C}\widetilde{A}^{-1}
\widetilde{B}\right) \cdot\det\widetilde{A} \cdot\det D\\
&=\(\lambda-2+\left| q_{11} \right|^2\)^2\(\lambda-1\)^{n-3}\lambda
\(\lambda-1\)^{n-1}\lambda^{\(n-1\)^2}\\
&=\(\lambda-2+\left| q_{11} \right|^2\)^2\(\lambda-1\)^{2n-4}\lambda^
 {\(n-1\)^2+1}.
\end{aligned}$$
Finally we observe that $q_{11}=\Tr X$. The proof is complete.
\end{proof}

Immediately we obtain
\begin{cor}\label{corrankX1}
Let $X$ be a complex square matrix of order $n$ $(\geq2)$ with $\|X\|=1$ and $rank(X)=1$. Then $\lambda_1(T_X)=2$ if and only if $\Tr(X)=0$.
\end{cor}

\begin{rem}
Actually, the conditions $\|X\|=1$, $\rank(X)=1$ and $\Tr(X)=0$ in Corollary \ref{corrankX1} implies that $X$ is unitary similar to $\diag(X_0,O)$, where
$$X_0=
\begin{pmatrix}
    0 &    0    \\
    1 &    0   \\
\end{pmatrix}.$$
Here, we give a simple calculation. Suppose $X=Q_1\Lambda Q_2$ is the singular value decomposition of $X$ and $Q=Q_2Q_1$, then  $Q_1^*XQ_1=\Lambda Q$. Due to $\|X\|=1$, $\rank(X)=1$ and $\Tr(X)=0$, we can assume
$$\Lambda Q=
\begin{pmatrix}
    0 &    q    \\
    0 &    0   \\
\end{pmatrix},$$
where $q=(q_{12},q_{13},\cdots,q_{1n})$ and $\|q\|=1$. Extend $q$ to be a unit orthogonal basis $\{q,p_1,p_2,\cdots,p_{n-2}\}$ of $\mathbb{C}^{n-1}$ and let
$$U=
\begin{pmatrix}
    0 & 1& 0& 0&\cdots&  0    \\
   q^*&  O &   p_1^*&  p_2^*&\cdots&  p_{n-2}^* \\
\end{pmatrix},$$
then $U^*U=I$ and $U^*Q_1^*XQ_1U=\diag(X_0,O)$.
\end{rem}

The last special case of Theorem \ref{thm-special-cases} is a simple consequence of Corollary \ref{cortrTX}.
\begin{thm}
The complex LW Conjecture \ref{conj2D} is true when $n=2,3$.
\end{thm}
\begin{proof}
The case $n=2$ is an immediate consequence of Corollary \ref{cortrTX} since it implies the set of eigenvalues $\lambda(T_X)$ is weakly majorized by $\{2^2, 0^2\}$.

The case $n=3$ is similar, since Corollary \ref{cortrTX} shows that $$\sum_{i=1}^{2k}\lambda_i(T_X)\leq \Tr T_X=6-2|\Tr X|^2\leq6\leq 2k+2\quad \textit{for any }k\geq2,$$
and for $k=1$ it follows from the BW inequality (e.g., Theorem \ref{thm1}) that $$\sum_{i=1}^{2k}\lambda_i(T_X)=2\lambda_1(T_X)\leq4=2k+2.$$
The proof is complete.
\end{proof}

Now we come to prove the partial results Theorems \ref{thm-conj3-weak} and \ref{thm-conj7-weak}.\\
\textbf{Proof of Theorem \ref{thm-conj3-weak}}
By Lemma \ref{lem3} and the proof of Theorem \ref{thm1}, for the fixed sequence $i_1=2, i_2=3, i_3=4$, we have
$$\begin{aligned}
\sum_{i=2}^4\lambda_i(T_X)&\leq\sum_{i=2}^4\lambda_i(K_1)+\sum_{i=1}^3\lambda_i(K_2)\\
&\leq 3\Big({\sigma}_1^2(X)+{\sigma}_2^2(X)\Big)+2\sum_{i=1}^3\Big(\lambda_i(-A^t\otimes A)+\lambda_i(B^t\otimes B)\Big),
\end{aligned}$$
as $K_2=-X^t\otimes X^*-\overline{X}\otimes X=2\(B^t\otimes B-A^t\otimes A\)$ for the decomposition $X=A+B$ with $A$ Hermitian and $B$ skew-Hermitian. Similarly we have
 $$\lambda_1(T_X)=\lambda_2(T_X)\leq{\sigma}_1^2(X)+{\sigma}_2^2(X)+2\Big(\lambda_1(-A^t\otimes A)+\lambda_1(B^t\otimes B)\Big).$$
This implies
$$
\sum_{i=1}^4\lambda_i(T_X)\leq 4\Big({\sigma}_1^2(X)+{\sigma}_2^2(X)\Big)+\phi(X),$$
where $\phi(X):=\varphi(A)+\widetilde{\varphi}(B)$ and
$$\begin{aligned}&\varphi(A):=4\lambda_1(-A^t\otimes A)+2\sum_{i=2}^3\lambda_i(-A^t\otimes A),\\
&\widetilde{\varphi}(B):=4\lambda_1(B^t\otimes B)+2\sum_{i=2}^3\lambda_i(B^t\otimes B).\end{aligned}$$

Let $\lambda(A)=\{a_1, \cdots, a_n\}$, $a_1\geq\cdots\geq a_n$;
$\lambda(B)=\{b_1\mathbf{i}, \cdots, b_n\mathbf{i}\}$,
$b_1\geq\cdots\geq b_n$.
Then by Lemma \ref{lem2},
$$\lambda(-A^t\otimes A)=\{-a_ia_j: 1\leq i, j\leq n\}, $$
$$\lambda(B^t\otimes B)=\{-b_ib_j: 1\leq i, j\leq n\}. $$
We claim that $$\phi(X)=\varphi(A)+\widetilde{\varphi}(B)\leq \sqrt{10}\Big(\|A\|^2+\|B\|^2\Big).$$
We only need to show $\varphi(A)\leq \sqrt{10}\|A\|^2$ since the case for $\widetilde{\varphi}(B)$ is similar.
Obviously $\varphi(A)$ would be non-positive unless $a_1>0>a_n$, in which case we have $$\lambda_1(-A^t\otimes A)=\lambda_2(-A^t\otimes A)=a_1|a_n|=\max_{i,j}\{-a_ia_j\}$$
and we can also assume without loss of generality that $$\lambda_3(-A^t\otimes A)=\lambda_4(-A^t\otimes A)=a_2|a_n|\geq0.$$
Then
$$\begin{aligned}
\varphi(A)
&=6a_1|a_n|+2a_2|a_n|=2|a_n|(3a_1+a_2)\\
&\leq 2\sqrt{10}|a_n|\sqrt{a_1^2+a_2^2}\leq \sqrt{10}(a_n^2+a_1^2+a_2^2) \\
&\leq \sqrt{10}\|A\|^2.
\end{aligned}$$

In conclusion,
$$\sum_{i=1}^4\lambda_i(T_X)\leq 4\Big({\sigma}_1^2(X)+{\sigma}_2^2(X)\Big)+\sqrt{10}\Big(\|A\|^2+\|B\|^2\Big)\leq(4+\sqrt{10})|X\|^2.$$
By Proposition \ref{prop2}, this completes the proof.
\hfill$\Box$

\begin{rem}
From the proof, one can see that the (non-sharp) upper bounds for the complex version and real version of Conjecture \ref{conj3} are no different, both $2+\sqrt{10}/2$.
\end{rem}

\begin{rem}
%The main cause of calculation error is from the relation of the singular values and the eigenvalues of Hermitian(skew-Hermitian) part of a matrix.
%The Hermitian matrix is divided into three Hermitian matrices to estimate separately so that it causes more error.
The reason of why we did not get the optimal upper bound $3$ of Conjecture \ref{conj3} mainly comes from that we divided the Hermitian matrix $K_X^*K_X$ into three parts and estimated them separately. The following example explains that the upper bound $2+\sqrt{10}/2$ we got in this way cannot be optimal. Set
%It can't expect this method to solve the conjecture \ref{conj3}.
%There is an example:
$$X=
\begin{pmatrix}
   -0.1236 &   0.0334 &    0.0647 \\
   -0.4343  & 0.1029  & -0.8833 \\
         0      &   0   &      0 \\
\end{pmatrix}.
$$
By numerical calculation we see $$
%\begin{aligned}
\sum_{i=1}^4\lambda_i(T_X)\approx 5.9814<6< 4\Big({\sigma}_1^2(X)+{\sigma}_2^2(X)\Big)+\phi(X)\approx 7.0554<4+\sqrt{10}.
%\end{aligned}
$$
\end{rem}

%\begin{rem}
%The upper bound $2+\sqrt{2}$ of the complex version conjecture \ref{conj3} would be obtained when
% $\sigma^2_1(X)+\sigma^2_3(X)\leq2\sigma^2_2(X)$.
%\end{rem}

To estimate higher order eigenvalues, we need the following lemma.
\begin{lem}\label{numbers inequality}
Suppose $\eta_1,\eta_2,\cdots,\eta_{n_1}$ and
$\omega_1,\omega_2,\cdots,\omega_{n_2}$ are nonnegative real numbers and $r_{ij}\in \{0,1\}$ such that
$$\sum_{i=1}^{n_1}\eta_i^2+\sum_{i=1}^{n_2}\omega_i^2=1,\quad
\sum_{i=1}^{n_1}\sum_{j=1}^{n_2}r_{ij}=m.$$
Then we have
\begin{equation}
\sum_{i=1}^{n_1}\sum_{j=1}^{n_2}\eta_i\omega_{j}r_{ij}\leq
\frac{\sqrt{m}}{2}.
\end{equation}
\end{lem}

\begin{proof}
Suppose $\eta_1\geq\cdots\geq\eta_{n_1}\geq0$ and $\omega_{1}\geq\cdots\geq\omega_{n_2}\geq0$, without loss of generality we can select the following $m$ elements with non-vanishing $r_{ij}$'s:
\begin{itemize}
\item $\eta_1\omega_1\geq\eta_1\omega_2\geq
\cdots\geq\eta_1\omega_{p_1}$
\item $\eta_2\omega_1\geq\eta_2\omega_2\geq
\cdots\geq\eta_2\omega_{p_2}$
\item $\cdots$
\item $\eta_t\omega_1\geq\eta_t\omega_2\geq
\cdots\geq\eta_t\omega_{p_t}$
\end{itemize}
where $p_1+p_2+\cdots+p_t=m$. Thus we complete the proof by
$$\begin{aligned} \sum_{i=1}^{n_1}\sum_{j=1}^{n_2}\eta_i\omega_{j}r_{ij}=\sum_{i=1}^{t}\eta_i\sum_{j=1}^{p_i}\omega_j
&\leq\sqrt{\sum_{i=1}^{t}\eta_i^2}\sqrt{\sum_{i=1}^{t}(\sum_{j=1}^{p_i}\omega_j)^2}
\leq\sqrt{\sum_{i=1}^{t}\eta_i^2}\sqrt{\sum_{i=1}^{t}p_i\sum_{j=1}^{p_i}\omega_j^2} \\
&\leq\sqrt{\sum_{i=1}^{t}\eta_i^2}\sqrt{\sum_{i=1}^{t}p_i\sum_{j=1}^{n_2}\omega_j^2}
\leq\sqrt{\sum_{i=1}^{n_1}\eta_i^2}\sqrt{m\sum_{j=1}^{n_2}\omega_j^2} \\
&\leq\frac{\sqrt{m}}{2} (\sum_{i=1}^{n_1}\eta_i^2+\sum_{i=1}^{n_2}\omega_i^2)=\frac{\sqrt{m}}{2}. \qquad \qquad \qquad \qquad \qedhere
\end{aligned}$$
\end{proof}
\textbf{Proof of Theorem \ref{thm-conj7-weak}.}
The proof is similar to that of Theorem \ref{thm-conj3-weak}. Briefly, by Lemma \ref{lem3} and Lemma \ref{numbers inequality}, we have
$$\begin{aligned}
\sum_{i=1}^{2k}\lambda_i(T_X)
&\leq\sum_{i=1}^{2k}\lambda_i(K_1)+\sum_{i=1}^{2k}\lambda_i(K_2)\\
&\leq\sum_{i=1}^{2k}\lambda_i(K_1)+
2\sum_{i=1}^{2k}\Big(\lambda_i(-A^t\otimes A)+\lambda_i(B^t\otimes B)\Big)\\
&\leq2k+1+2\Big(\sqrt{k}\|A\|^2+\sqrt{k}\|B\|^2\Big)\\
&=2k+1+2\sqrt{k},
\end{aligned}$$
where $\sum_{i=1}^{2k}\lambda_i(K_1)\leq2k+1$ follows from $$\lambda_1(K_1)=2\sigma_1^2(X)\leq2, \quad \lambda_i(K_1)\leq \lambda_2(K_1)=\sigma_1^2(X)+\sigma_2^2(X)\leq1 \textit{ for } i\geq2;$$
and $$\sum_{i=1}^{2k}\lambda_i(-A^t\otimes A)\leq2\sum_{r=1}^{k}\lambda_{2r-1}(-A^t\otimes A)\leq \sqrt{k}\|A\|^2$$ (similar for $\sum_{i=1}^{2k}\lambda_i(B^t\otimes B)\leq \sqrt{k}\|B\|^2$) follows by setting in Lemma \ref{numbers inequality}
$$\Big\{\begin{array}{ll}
\eta_i:=a_i/\|A\|,& 1\leq i\leq n_1,\\
\omega_j:=-a_{n_1+j}/\|A\|, & 1\leq j\leq n-n_1,
\end{array}$$
for  $a_1\geq\cdots\geq a_{n_1}\geq0\geq a_{n_1+1}\geq\cdots\geq a_n$ and noticing that now the nonnegative eigenvalues $\lambda_{2r-1}(-A^t\otimes A)=\lambda_{2r}(-A^t\otimes A)=-a_ia_{n_1+j}$ appear in pairs.
\hfill $\Box$

\begin{ack}
The authors would like to thank David Wenzel for his valuable discussions and comments.
\end{ack}

%%%%%%%%%%%%%%%%%%%%%%%%%%%%%%%%

\end{document}